\documentclass[leqno,12pt]{article}
\usepackage{preamble}
\title{\vspace{-1em}\Pnormal{}  categories}
\author{
  Sandra Mantovani\textsuperscript{*}
  \\
  \texttt{sandra.mantovani@unimi.it}
  \and
  Mariano Messora\textsuperscript{*}
  \\
  \texttt{mariano.messora@unimi.it}
}
\date{
\vspace{-.5em}
\footnotesize{\textsuperscript{*}Department of Mathematics, University of Milan, Via Cesare Saldini 50, 20133 Milan, Italy}
\vspace{-1em}
}
\begin{document}
\maketitle
\hrule\vspace{.7em}
\noindent 
\thispagestyle{empty}%
\noindent\textbf{Abstract.} In this paper we introduce the notion of (pointed) \emph{\pnormal{} category}, modelled after regular 
categories, but with the key notions of coequaliser and kernel pair replaced by those of cokernel and kernel. This framework provides a natural setting for extending certain classical results in algebra. We study the fundamental properties of \pnormal{} categories, including a characterisation in terms of a factorisation system involving normal epimorphisms, and a categorical version of Noether's so-called `third isomorphism theorem'. We also present a range of examples, with the category of commutative monoids constituting a central one. In the second part of the paper we extend \pnormal{}ity and its related properties to the non-pointed context, using kernels and cokernels defined relative to a distinguished class of trivial objects.
\vspace{.7em}
\hrule
\vspace{.7em}
\vspace{.7em}
\hrule
\renewcommand{\baselinestretch}{0.8}\normalsize
\tableofcontents
\renewcommand{\baselinestretch}{1}\normalsize
\vspace{.7em}
\bigskip
\hrule
\vspace{.5em}
\noindent Accepted for publication in \emph{Applied Categorical Structures}
\hfill December 2025
\restoregeometry
\section*{Introduction}
\addcontentsline{toc}{section}{Introduction}
\label{Intro}
It is a straightforward observation that in any additive category, regular epimorphisms and normal epimorphisms define the same class of maps: indeed, the coequaliser of a pair of morphisms is also the cokernel of their difference. Similarly, in this setting, a map is a monomorphism if and only if its kernel is trivial.
Consequently, if an additive category is also regular, normal epimorphisms and morphisms with a trivial kernel form a stable factorisation system, which coincides with the  one determined by regularity. This phenomenon -- in which normal epimorphisms coincide with regular epimorphisms and are pullback-stable -- is not limited to regular additive categories; it also occurs, for instance, in the category of groups, and, more in general, in any homological category (\cite{BBBOOK}), as well as in several other well-studied contexts, where it is known to give rise to a range of useful properties. These recurring features, along with the properties they entail, have motivated the definition of \emph{normal categories} (\cite{NORMAL}), denoting regular pointed categories in which every regular epimorphism is normal.

Although the notion of normal category indeed applies to numerous commonly encountered settings, it frames the behaviour of normal epimorphisms only in relation to regularity. 
In contrast, we wish to emphasize that many structural properties of normal categories can persist even outside the context of regular categories. In particular,
we observe that normal epimorphisms and maps with a trivial kernel can form a stable factorisation system entirely independently of any connection to the one given by regular epimorphisms and monomorphisms. A remarkable
example illustrating this independence is provided by the category of commutative monoids, where one can show that the two factorisation systems just mentioned coexist, are both stable, and yet  remain distinct, leading to several interesting consequences (see \zcref{mon-not-norm}). 
This example, along with others, suggests that the stability of normal epimorphisms can be meaningfully studied on its own, without relying on regularity or any identification with regular epimorphisms. 

Building on the above considerations, we propose a definition of (pointed) \emph{\pnormal{} category}, modelled on the classical definition of regular category, but where the key notions of coequaliser and kernel pair are replaced by those of cokernel and kernel: a \pnormal{} category is thus defined as a pointed, finitely complete category where cokernels of kernels exist and normal epimorphisms are pullback-stable (\zcref{pnormal-def}). This notion is strictly weaker than that of normal category (\cite{NORMAL}), as proven by the example of commutative monoids discussed above, and provides a natural setting for extending certain classical results in algebra. The aim of this paper is then to study \pnormality{} and its associated categorical properties in full generality,
irrespective of
regularity or any specific behaviour of regular epimorphisms. 
In particular, we prove that many properties that hold for regular categories can be carried over to \pnormal{} categories by replacing regular epimorphisms with normal epimorphisms, and monomorphisms with morphisms with a trivial kernel. Notably, we find a characterisation of \pnormal{} categories as those pointed, finitely complete categories where normal epimorphisms form (with maps with a trivial kernel) a stable factorisation system (\zcref{pnorm-fs}). Other interesting properties in \pnormal{} categories include the fact that pullback squares of normal epimorphisms are pushouts (\zcref{pb-is-po-pointed}), a pullback cancellation principle involving normal epimorphisms  (\zcref{pointed-pf-pb}) and a categorical version of Noether's so-called `third isomorphism theorem' (\zcref{Noether-3}). We also present new examples that fit within this broad framework, further exhibiting the independence of \pnormality{} and regularity (the two can coexist and be distinct, one can exist without the other, and so on).

The second part of this paper is devoted to a  generalisation of \pnormal{} categories to the non-pointed setting. A key motivation for this extension arises from the behaviour of \emph{slices}: a well-known and useful property of regular categories is their stability under taking slices. 
It is then natural to ask how \pnormal{} categories behave in this regard. However, the slice of a pointed category is generally no longer pointed, and as a result, the study of slices of \pnormal{} categories naturally leads us beyond the realm of pointed categories. To pursue this line of inquiry, we develop a non-pointed generalisation of \pnormality{} by 
adopting the standard approach of defining kernels and cokernels relative to some distinguished class $\tobj$ of trivial objects, rather than with respect to a zero object. This leads to a notion of \emph{\normal{} category} (\zcref{normal-def}), which reduces to standard \pnormality{} when the ambient category is pointed and $\tobj$ is the class of zero objects. 

This generalised approach indeed allows us to recover a form of slice stability: the slice $\slice C X$ of a \normal{} category $\cat C
$ over an object $X$ is naturally a $(\slice ZX)$-\pnormal{} category (\zcref{slice-is-normal}). Moreover, many of the key properties of \pnormal{} categories are retained in this broader setting, most notably the characterisation in terms of a stable factorisation system (\zcref{characterisation-fs}). These and other properties, along with the additional subtleties involved with working in a non-pointed setting, are investigated in \zcref{non-p-normal,sec-normal-fs,sec-properties,sec-ex-seq}. This non-pointed generalisation also greatly broadens the range of possible examples, with a detailed list provided in \zcref{Examples}. Moreover, this framework appears to be particularly well-suited to the study of certain characterisations of \emph{pretorsion theories} (see \cite{PRETORSION,HTT}) -- a direction that will be addressed in future work.

Other interesting approaches to both pointed and non-pointed  `normality' exist in the literature, including \cite{GRANDIS13,GRAN12,VOLPE,THOLEN,THOLEN25}.
In \cite{GRAN12,VOLPE}, the notions considered require the ambient category to be regular, while \cite{THOLEN} develops a notion of normality relative to a given proper stable factorisation system. In \cite{THOLEN25} the authors provide non-pointed notions of normal monomorphism and normal epimorphism that are not based on freely chosen `trivial objects', but are instead defined naturally in terms of initial and terminal objects in any category with enough limits and colimits.
Accordingly, the premises of these four approaches differ significantly from the one adopted in the present paper. In \cite{GRANDIS13}, the author considers an axiom of partial stability of (possibly non-pointed) normal epimorphisms, along with other additional hypotheses. \zcref[capfirst=true]{weakly-normal} is devoted to a comparison with this last notion.
\section{Preliminaries}
In this section we establish some notation and we recall a few known results which will be useful later. We begin by setting up the notation for certain special classes of arrows.
\begin{notation}
For any category $\cat C$ we use the following notation:
\begin{itemize}
    \item $\StrEpi$ denotes the class of strong epimorphisms in $\cat C$.
    \item $\RegEpi$ denotes the class of regular epimorphisms in $\cat C$. These are graphically represented using the special arrow `$\regepito$'.
    \item $\Mono$ denotes the class of monomorphisms in $\cat C$. These are graphically represented using the special arrow `$\monoto$'.
\end{itemize}
We also represent (ordinary) epimorphisms with the special arrow `$\epito$'. If moreover $\cat C$ is pointed we introduce the  following additional notation:
\begin{itemize}
    \item $\pNEpi$ denotes the class of normal epimorphism in $\cat C$. These are graphically represented using the special arrow `$\nepito$'.
    \item $\pNMono$ denotes the class of normal monomorphisms in $\cat C$. These are graphically represented using the special arrow `$\nmonoto$'.
    \item $\pTKer$ denotes the class of maps in $\cat C$ whose kernel is reduced to zero. These are graphically represented using the special arrow `$\tkerto$'.
\end{itemize}
\end{notation}
{
\newcommand{\lcl}{\mathscr{E}}
\newcommand{\rcl}{\mathscr{M}}
Next, we fix the terminology for factorisation systems that we will use throughout this paper. See for example \cite{JOY} for details.
\begin{definition}
Let $\cat C$ be any category. By \emph{factorisation system on $\cat C$} we mean a pair $(\lcl,\rcl)$ of classes of morphisms in $\cat C$ satisfying the following conditions:
\begin{enumerate}
    \item both $\lcl$ and $\rcl$ are closed under composition with isomorphisms;
    \item every map $f\colon X\to Y$ in $\cat C$ factors as $f=\comp e m$, with $e\in\lcl$ and $m\in\rcl$;
    \item every $f\colon X\to Y$ in $\lcl$ satisfies the \emph{left lifting property} with respect to every $f'\colon X'\to Y'$ in $ \rcl$, i.e.\ given any $x\colon X\to X'$ and any $y\colon Y\to Y'$ such that $\comp fy=\comp x{f'}$, there exists a unique $d\colon Y\to X'$ in $\cat C$ such that $\comp fd=x$ and $\comp d{f'}=y$ (we also say that $f'$ satisfies the \emph{right lifting property} with respect to $f$).
    \[
    \begin{tikzcd}
        X\arrow[r, "x"]\arrow[d, "f"'] & X'\arrow[d, "f'"]
        \\
        Y\arrow[r, "y"']\arrow[ur, "d"{description}] & Y'
    \end{tikzcd}
    \]
\end{enumerate}
(This is sometimes referred to as an \emph{orthogonal factorisation system}.)

Furthermore, the factorisation system $(\lcl,\rcl)$ is said to be \emph{stable} if the pullback of a map in $\lcl$ along any map in $\cat C$ is again in $\lcl$.   
\end{definition}
Classes of maps involved in factorisation systems enjoy many useful properties. We recall a few of them that we will need in the next sections.
\begin{proposition}
\label{fs-properties}
    Let $\cat C$ be a category and let $(\lcl,\rcl)$ be a factorisation system on $\cat C$. Then the following properties hold.
    \begin{enumerate}
        \item\label{fs-property-iso} If a map in $\cat C$ is both in $\lcl$ and $\rcl$, then it is an isomorphism.
        \item\label{fs-property}$\lcl$ is the class of maps in $\cat C$ satisfying the left lifting property with respect to all the maps in $\rcl$, and, dually, $\rcl$ is the class of maps in $\cat C$ satisfying the right lifting property with respect to all the maps in $\lcl$.
        \item\label{fs-property-cancel}If some composable maps $f$ and $g$ in $\cat C$ are such that $\comp f g$ and $f$ lie in $\lcl$, then $g\in\lcl$ as well. Dually, if $\comp fg$ and $g$ lie in $\rcl$, then $f\in\rcl$ as well.
    \end{enumerate}
\end{proposition}
}
\section{The pointed case: \pnormal{} categories}
\label{pointed-case}
We now begin the main development of the paper by formally stating the definition of \pnormality{}, initially within the pointed setting, followed by a presentation of its fundamental properties.
While a more general, non-pointed version will be presented later, the pointed setting provides a more intuitive framework with less technicalities and includes many interesting examples, making it a natural starting point.

As previously discussed, the notion of \pnormality{} is based on that of regular categories, with cokernels taking the place of coequalisers and kernels that of kernel pairs. 
\begin{definition}
\label{pnormal-def}
    A \emph{\pnormal{}} category is a pointed category $\cat C$ such that:
    \begin{enumerate}
        \item $\cat C$ has finite limits;
        \item $\cat C$ has cokernels of kernels;
        \item $\cat C$ has pullback-stable normal epimorphisms.
    \end{enumerate}
\end{definition}
Let us also introduce the weaker notion of a \emph{\wpnormal{}} category, which is closely related to axioms considered in \cite{GRANDIS13} (see also \zcref{weakly-normal}).
\begin{definition}
A \emph{\wpnormal{}} category is a pointed category $\cat C$ such that:
\begin{enumerate}
    \item $\cat C$ has pullbacks along normal monomorphisms;
    \item $\cat C$ has cokernels of kernels;
    \item the pullback of a normal epimorphism along a normal monomorphism is again a normal epimorphism.
\end{enumerate}
\end{definition}
We also recall the definition of a \emph{normal} category from \cite{NORMAL}, as it will be referenced throughout the text.
\begin{definition}
    A \emph{normal} category is a pointed regular category $\cat C$ such that every regular epimorphism in $\cat C$ is a normal epimorphism. 
\end{definition}
With these definitions in place, we now proceed to establish a number of general properties of (semi-)\pnormal{} categories. Most proofs are deferred until  \zcref{non-p-normal}, where we study \pnormality{} in greater generality. Many of these properties mirror familiar results from the theory of regular categories (for example in \cite{BORCEUX94b,GRAN21}), but are new in this context. 
\begin{proposition}
\label{pnorm-fs}
    Let $\cat C$ be a pointed category with finite limits. The following are equivalent:
    \begin{enumerate}[(i)]
        \item\label{p-equiv-def-1} $\cat C$ is \pnormal{};
        \item\label{p-equiv-def-2} $\cat C$ admits \pfacsys{} as a stable factorisation system;
     \item $\cat C$ admits a stable factorisation system of the form $\bigl(\pNEpi,\mathscr M\bigr)$, for some class $\mathscr M$ of morphisms in $\cat C$;
     \item $\cat C$ admits a stable factorisation system of the form $\bigl(\mathscr E,\pTKer\bigr)$, for some class $\mathscr E$ of morphisms  in $\cat C$.
     \end{enumerate}
\end{proposition}

\begin{proposition}
\label{pb-is-po-pointed}
    In a \pnormal{} category, consider the following pullback square. 
    \begin{equation*}
    \begin{tikzcd}
    \cdot\arrow[r,"f'"]\arrow[d,"x"']&\cdot\arrow[d,"y"]
    \\
    \cdot\arrow[r,"f"']&\cdot
    \end{tikzcd}
    \end{equation*}
    The above diagram is also a pushout whenever one of the following conditions hold. \begin{enumerate*}[(a)]
        \item $f$ is a normal epimorphism. \item $x$ is a normal epimorphism, $f$ is a regular epimorphism, and $f'$ is an epimorphism. 
    \end{enumerate*}
\end{proposition}
\begin{proposition}
\label{pointed-pf-pb}
In a \pnormal{} category, consider then the following commutative diagram, where $f'$ is a normal epimorphism and both the outer rectangle and left-hand square are pullbacks. 
\begin{equation*}
\begin{tikzcd}[column sep = 3.7em]
    A\arrow[r,"f"]\arrow[d, "a"'] & B\arrow[r, "g"]\arrow[d,"b"'] & C\arrow[d, "c"']
    \\
    A'\arrow[r, "f'"', normepi] & B'\arrow[r, "g'"'] &C'
\end{tikzcd}
\end{equation*}
Then the right-hand square is also a pullback.
\end{proposition}
\begin{proposition}
\label{ex-pnorm}
    Let $\cat C$ be a \pnormal{} category, and consider the following commutative diagram where $f$ is the kernel of $g$. \begin{equation*}
    \begin{tikzcd}
        A'\arrow[r,"{f'}"]\arrow[d,"a"']&B'\arrow[r,"{g'}"]\arrow[d,"b"']&C'\arrow[d,"c"']
        \\
        A\arrow[r, "f"', normmono]& B\arrow[r,"g"']& C
    \end{tikzcd}
    \end{equation*}
     If the left square is a pullback and $g'$ is the cokernel of $f'$, then $c\in\pTKer$.
    Conversely, if $c\in\pTKer$ and $f'$ is the kernel of $g'$, then the left square is a pullback.
\end{proposition}
\begin{proposition}
In a \pnormal{} category, consider the following commutative diagram where $f$ is a normal epimorphism, $(r_0,r_1)$ is the kernel pair of $f$ and $(s_0,s_1)$ is the kernel pair of $g$.
\begin{equation*}
   \begin{tikzcd}
        \cdot \arrow[r, yshift=.3em, "r_0"]\arrow[r,yshift=-.3em, "r_1"']\arrow[d]& \cdot\arrow[r,normepi, "f"] \arrow[d] & \cdot\arrow[d]
        \\
        \cdot \arrow[r, yshift=.3em, "s_0"]\arrow[r,yshift=-.3em, "s_1"']& \cdot\arrow["g"',r] & \cdot
    \end{tikzcd}
\end{equation*}
If either of the left-hand squares is a pullback, then the right-hand square is a pullback.
\end{proposition}
\begin{proposition}
    Let $\cat C$ be a \pnormal{} category. Then the following properties hold.
    \begin{enumerate}
        \item The product of two normal epimorphisms is a normal epimorphism. Furthermore, the product of two exact sequences is again an exact sequence.
        \item Let $E\colon A\to B\to C$ be an exact sequence, and let $E'\colon A'\to B'\to C'$ be the sequence obtained by pulling back $E$ along a given map $c\colon C'\to C$. Then $E'$ is exact if and only if $c\in\pTKer$.
    \end{enumerate}
\end{proposition}
\begin{proposition}
\label{Noether-3}
    In a \pnormal{} category, Noether's third isomorphism theorem holds: given normal monomorphisms $m\colon M\nmonoto A$ and $n\colon N\nmonoto A$, with $n$ factoring through $m$, then the induced map on the cokernels $N/M\to A/M$ is a normal monomorphism whose cokernel $(A/M)/(N/M)$ is canonically isomorphic to the cokernel $A/N$ of $n$.
\end{proposition}
\begin{proposition}
\label{fun-p}
    Let $\cat C$ be a \pnormal{} category. For any category $\cat B$, the functor category $\Fun B C$ is \pnormal{}.
\end{proposition}
\begin{remark}
    Some of the properties of prenormal categories actually extend to semi-prenormal categories (in particular Propositions~\ref{ex-pnorm}, \ref{Noether-3}, \ref{fun-p} and a version of \zcref{pnorm-fs}; see \zcref{weakly-normal} for more details).
\end{remark}
\begin{remark}
    As mentioned in the introduction, the slice of a pointed category is not pointed, in general, and therefore the slice of a \pnormal{} category is in general not \pnormal{}. It is however \pnormal{} in the non-pointed sense discussed in \zcref{non-p-normal} (see \zcref{slice-is-normal}). 
\end{remark}

\Pnormality{} is in general independent of both regularity and of normality in the sense of \cite{NORMAL}, as illustrated in \zcref{cmon-is-normal,preord-cmon-is-normal,mon-not-norm,pocmon-is-normal}. However, the three notions are interconnected, as established in \zcref{protonormal-protomodular,pointed-pnormal-vs-normal,abelian-implies} below.
\begin{proposition}
\label{pointed-pnormal-vs-normal}
    Let $\cat C$ be a \pnormal{} category. Then the following are equivalent:
    \begin{enumerate*}[(i)]
        \item $\cat C$ is normal in the sense of \cite{NORMAL};
        \item $\StrEpi=\pNEpi$;
        \item $\RegEpi=\pNEpi$;
        \item $\Mono=\pTKer$.
    \end{enumerate*}
\end{proposition}
\begin{proposition}
\label{protonormal-protomodular}
Let $\cat C$ be a pointed protomodular category (as in \cite{BBBOOK}). Then the following are equivalent:
\begin{enumerate*}[(i)]
    \item $\cat C$ is \pnormal{};
    \item $\cat C$ is homological;
    \item $\cat C$ is normal (in the sense of \cite{NORMAL});
    \item $\cat C$ is regular.
\end{enumerate*}
\end{proposition}
\begin{proof}
 In the presence of protomodularity we have
\[
\Mono=\pTKer, \qquad \RegEpi=\pNEpi.
\]
Therefore \pnormality, normality and regularity are the same notion in the presence of protomodularity. Then $\cat C$ is homological if it is protomodular and regular, by definition.
\end{proof}

We devote the rest of the section to a collection of various examples and non-examples.
\begin{example}
\label{abelian-implies}
    We have the following chain of implications:
    \begin{equation*}
        \text{abelian}\implies\text{semi-abelian}\implies\text{homological}\implies\text{normal}\implies\text{\pnormal{}}\implies \text{\wpnormal{}}.
    \end{equation*}
    Therefore any abelian, semi-abelian, homological or normal category is \pnormal{}.
\end{example}
\begin{example}
\label{cmon-is-normal}
{
The category $\catCMon$ of commutative monoids is \pnormal{}.

As is well known, $\catCMon$ is pointed, complete and cocomplete. It is also a regular category and regular epimorphisms are simply surjective monoid homomorphisms. Given a kernel $k\colon K\inclusion M$, the cokernel of $k$ is obtained by taking the quotient of $M$ by the congruence
\begin{equation*}
    x\sim y\textnormal{ if } x+a=y+b\textnormal{ for some } a,b\in K,
\end{equation*}
so that normal epimorphisms can be characterised as those surjective monoid homomorphisms $f\colon M\to N$ with the following additional property:
\begin{equation}
\zcsetup{reftype=property}
\label{cmon-normepi}    \mbox{}\hfil \begin{minipage}[m]{0.7\linewidth}
\centering
      for every $x,y\in M$, if $f(x)= f(y)$ then $x+a=y+b$ for some $a,b\in M$ such that $f(a)=f(b)=0$. 
\end{minipage}\hfil\mbox{}
\end{equation}
It is a simple check to verify that surjective maps satisfying \zcref{cmon-normepi} are stable under pullback. Therefore $\catCMon$ is indeed \pnormal{}.

Notice that \zcref{cmon-normepi} is not automatically satisfied by regular epimorphisms. Consider for example the operation of sum of natural numbers, seen as a map of commutative monoids
\begin{equation}
\label{Nsum}
    +\colon \naturals\times\naturals\to\naturals.
\end{equation}
Since this is a surjective map not satisfying \zcref[capfirst=true]{cmon-normepi}, we conclude that regular epimorphisms in $\catCMon$ do not coincide with normal epimorphisms, and therefore, $\catCMon$ is a regular pointed category which is \pnormal{} but not normal in the sense of \cite{NORMAL}. We then have two different stable factorisation systems: the familiar regular epi-mono factorisation system, plus the factorisation system described in \zcref{pnorm-fs}. The above map \ref{Nsum} is surjective and has a trivial kernel: it is thus in the left part of the first factorisation system and in the right part of the second.
}
\end{example}
{
\newcommand{\fm}[1]{\langle {#1}\rangle}
\newcommand{\fmr}[2]{\langle {#1}\,\vert\,{#2\rangle}}
\newcommand{\ew}{0}
\begin{example}
\label{mon-not-norm}
The category $\catMon$ of not necessarily commutative monoids is a variety of algebras which is \emph{not} \pnormal{} -- and not even \wpnormal{}.

Consider the following short exact sequence of monoids
\begin{equation*}
\begin{tikzcd}
K\arrow[r,hookrightarrow] & M\arrow[r, "p"] &N,
\end{tikzcd}
\end{equation*}
where $M=\fmr {x,y,z}{x^2=y^2}$ is the quotient of the free monoid on the generators $x$, $y$ and $z$ by the relation $x^2=y^2$, and $K=\fm z$ is the free monoid generated by $z$, seen as a submonoid of $M$. Clearly we have that $N=M/K=\fmr {x,y}{x^2=y^2}$. Consider now the submonoid $N'$ of $N$ consisting of those elements in $N$ where only an even number of $x$s appear (this is well-defined, because the relation $x^2=y^2$ does not change the parity of the number of $x$s). It is easy to check that $N'$ is a normal submonoid of $N$ (it is the kernel of the map $N\to\fmr{x}{x^2=\ew}$ given by $x\mapsto x$ and $y\mapsto\ew$). Call $M'$ the pullback of $p$ along the inclusion $N'\hookrightarrow N$. Obviously $M'$ consists of those elements of $M$ where only an even number of $x$s appear.

We claim that the pullback projection $p'\colon M'\to N'$ is not a normal epimorphism. The kernel of $p'$ is still $K$, but we show that the kernel pair of $p'$ is not the smallest equivalence relation $R'$ on $M'$ that satisfies $(w,0)\in R'$ for all $w\in K$. This smallest equivalence relation, $R'$, can be explicitly described as the transitive closure of the following relation $\sim$:
\begin{equation}
\label{big-rel}
\begin{minipage}[m]{0.8\linewidth}
for $a,a'\in M'$, we say $a\sim a'$ if we we can write 
\begin{itemize}
    \item $a=u_1k_1\cdots u_n k_n$,
    \item $a'=u_1'k_1'\cdots u_n'k_n'$,
    \item $u_1h_1\cdots u_nh_n=u_1'h_1'\cdots u_n' h_n'$,
\end{itemize}
for some $n\in\naturals$, $u_i, u_i'\in M'$ and $k_i,k_i',h_i,h_i'\in K$, for all $i\in\{1,\dots n\}$.
\end{minipage}
\end{equation}
\newcommand{\maxn}[2]{\chi_{#1}({#2})}Consider $a=xzx\in M'$ and $b=yzy\in M'$. Clearly $p'(b)=y^2=x^2=p'(a)$, so that $(a,b)$ is in the kernel pair of $p'$. However, one can show that $(a,b)\notin R'$. The idea of the proof is as follows. For each $w\in M'$ define  $\maxn xw$, the maximal number of $x$s that can appear in $w$ (for any $w\in M'$, $w$ can be rewritten using the relation $x^2=y^2$ in a finite number of ways, and so there is a way to write $w$ such that the number of $x$s is maximal). Similarly, define $\maxn yw$ for each $w\in M'$. By using the explicit description \ref{big-rel} of $\sim$ and the fact that the $u_i$s in \ref{big-rel} cannot contain an odd number of $x$s, one proves that if $w_1,w_2\in M'$ are such that $w_1\sim w_2$, $\maxn x{w_1}=2$ and $\maxn y{w_1}=0$, then $\maxn x{w_2}=2$ and $\maxn y{w_2}=0$ as well. Applying this to $a$, one finds that if it was the case that $(a,b)\in R'$, then $b$ would satisfy $\maxn y b=0$. But clearly $\maxn y b =2$. 
\end{example}
}
\begin{example}
\label{omega-cmon}
{
    The example of commutative monoids \ref{cmon-is-normal} can be generalised as follows. 
    Consider a classical algebraic theory $\Omega=(\Sigma,\Tau)$ consisting of a set $\Sigma$ of operations of finite positive arities and a set $\Tau$ of equational axioms.
    In analogy with $\Omega$-groups, we consider the notion of a (commutative) $\Omega$-monoid, defined as a quadruple $\bigl(M,+,0,(\sigma_M)_{\sigma\in\Sigma}\bigr)$ such that:
    \begin{itemize}
        \item $(M,+,0)$ is a (commutative) monoid;
        \item $\bigl(M,(\sigma_M)_{\sigma\in\Sigma}\bigr)$ is an $\Omega$-algebra -- i.e.\ for each $n$-ary  operation $\sigma\in\Sigma$, $\sigma_M\colon M^n\to M$ is a (set-theoretic) function, and the $\sigma_M$s satisfy the axioms specified by $\Tau$;
        \item for each $n$-ary operation $\sigma\in\Sigma$, the function $\sigma_M\colon M^n\to M$ is a monoid homomorphism in each variable separately (in other words, the $\sigma_M$s are distributive over the base monoid operation and have the base monoid unit as an absorbing element).
    \end{itemize}
    
Examples of commutative $\Omega$-monoids for various $\Omega$s include commutative monoids themselves (where $\Omega$ is empty), non-unital semirings (where $\Omega$ consists of one binary operation satisfying associativity) and $R$-semimodules for a fixed semiring $R$ (with $\Omega$ containing a unary operation for every element of $R$ and some additional axioms).  

 For a fixed theory $\Omega$, consider $\catOm$, the variety of commutative $\Omega$-monoids (here `variety' is again understood in the sense of universal algebra). It is easily verified that in such a category, normal epimorphisms are characterised in the same way as in $\catCMon$, and therefore $\catOm$ is \pnormal{}. 

Note that, for instance, the category $\catRg$ of non-unital semirings is another example of a category that is \pnormal{} but not normal in the sense of \cite{NORMAL}, because any commutative monoid $(M,+,0)$ can be seen as a non-unital semiring $(M,+,0,\ast)$ with $\ast$ given by $x\,\ast\, y=0$ for every $x,y\in M$. So the same counterexample seen in \zcref{cmon-is-normal} for $\catCMon$ holds for $\catRg$.
}
\end{example}
{
\newcommand{\pcmon}{\cat P}
\newcommand{\mord}{\le}
\begin{example}
\label{preord-cmon-is-normal}
Consider the category $\pcmon=\catPreordCMon$ of preordered commutative monoids. Its objects are  quadruples $(M,+,0,\mord)$ (which we will usually denote by just $M$) where $(M,+,0)$ is a commutative monoid, $(M,\mord)$ is a preordered set, and the following holds for every $x,y,a\in M$:
\begin{equation*}
    x\mord y\implies x+a\mord y+a.
\end{equation*}
Morphisms in $\pcmon$ are functions that are monoid morphisms and order morphisms at the same time. We prove that this category is \pnormal{} but not regular.

First, $\pcmon$ is pointed and limits in $\pcmon$ exist and are computed as in $\catCMon$. Given a kernel $k\colon K\inclusion M$ in $\pcmon$, the cokernel of $k$ is obtained by taking the cokernel $q\colon M\to Q$ in $\catCMon$ and equipping $Q$ with the preorder $\mord_Q$ given by
\begin{equation*}
\begin{minipage}[m]{0.5\linewidth}
\centering
$u\mord_Q v$ if there exist $x,y\in M$ such that $q(x)=u$, $q(y)=v$ and $x\mord_M y$.
\end{minipage}\hfil\mbox{}
\end{equation*}
The fact that $\mord_Q$ is transitive follows from the characterisation in \ref{cmon-normepi} of normal epimorphisms of commutative monoids. In fact, suppose $u\mord_Q v$ and $v\mord_Q w$. Then there exist $x,y,y',z\in M$ such that $q(x)= u$, $q(y)=q(y')=v$, $q(z)=w$, $x\mord_M y$ and $y'\mord_M z$. Since $q(y)=q(y')$, we find that $y+a=y'+a'$ for some $a,a'\in K$. Thus we obtain
\begin{equation*}
    x+a\mord_M y+a=y'+a'\mord_M z+a'.
\end{equation*}
Since $q(x+a)=q(x)=u$ and $q(z+a')=q(z)=w$, we deduce that $u\mord_Q w$ by definition. The rest of the proof that $q\colon M\to Q$ is indeed the cokernel of $k$ is now easy to carry out. We can then characterise normal epimorphisms in $\pcmon$ as those maps $f\colon M\to N$ satisfying \zcref{cmon-normepi} as well as the following.
\begin{equation}
\zcsetup{reftype=property} \label{preordcmon-normepi}
\begin{minipage}[m]{0.7\linewidth}
\centering
If $u\mord_N v$ for some $u,v\in N$, then there exist $x,y\in M$ such that $f(x)=u$, $f(y)=v$ and $x\mord_M y$.
\end{minipage}\hfil\mbox{}
\end{equation}
\zcref[capfirst=true]{cmon-normepi,preordcmon-normepi} are together pullback-stable, and so we conclude that $\pcmon$ is \pnormal{}.

To show that $\pcmon$ is not regular consider the following pullback square in $\pcmon$,
\begin{equation*}
\begin{tikzcd}
    D\arrow[r, "{p'}"]\arrow[d, "{i'}"',hookrightarrow]& C\arrow[d,hookrightarrow, "i"]
    \\
    A\arrow[r, "p"',regepi] & B
\end{tikzcd}
\end{equation*}
where the objects and maps are defined as follows.
\newcommand{\sord}{\preceq}
\newcommand{\sjoin}{\curlyvee}
\newcommand{\mjoin}{\vee}
\newcommand{\join}{\operatorname{join}}
\begin{itemize}[-]
    \item $A$ is the set $\{0,1,1',2\}$. We equip $A$ with the partial order  $\mord$ defined by
    \[
    0\mord 1,\quad 1'\mord 2
    \]
    and with the binary operation $\sjoin$, defined
as the join with respect to the the total order $\sord$ given by
    \(
    0\sord1\sord1'\sord2.
    \)
    One can check that with these definitions, $(A,\sjoin,0,\mord)$ is a preordered commutative monoid.
    \item $B$ is the set $\{0,1,2\}$ equipped with the total order $0\mord1\mord2$ and the binary operation $\mjoin$, defined as the join with respect to the same order $\mord$ on $B$. Clearly $(B,\mjoin,0,\mord)$ is a preordered commutative monoid. 
    \item The map $p\colon A\to B$ is defined as 
    \[
    p(0)=0,\quad p(1)=p(1')=1,\quad p(2)=2.
    \]
    It easy easy to check that $p$ is a regular epimorphism in $\pcmon$ (it is the coequalizer of its kernel pair).
    \item $C$ is the subset $\{0,2\}$ of $B$. We equip $C$ with the restrictions of the operation $\mjoin$ and the preorder $\mord$ of $B$, so that we have a full inclusion $i\colon C\inclusion B$.
    \item Finally, $(D,p',i')$ is the pullback of $p$ along $i$. One can easily check that we can take $D$ to be the set $\{0,2\}$ with the discrete preorder and the operation
    \[
    0+0=0,\quad 0+2=2+ 0=2,\quad 2+2=2;
    \] $i'$ is the inclusion of \{0,2\} into $\{0,1,1',2\}$ and $p'$ is the identity on the underlying set $\{0,2\}$. 
\end{itemize}
Since $p'$ is a monomorphism but not an isomorphism ($D$ and $C$ are not isomorphic as preordered sets), it cannot be a regular epimorphism. But as $p$ is a regular epimorphism, we conclude that $\pcmon$ is not regular.
\end{example}
}
\begin{example}
\label{pocmon-is-normal}
The category $\catPOCMon$ of partially ordered commutative monoids, like $\catPreordCMon$, is \pnormal{} but not regular.

Limits in $\catPOCMon$ are computed in the same way as in $\catPreordCMon$.
As shown in \cite{FACCHINI20}, $\catPOCMon$ is a reflective subcategory of $\catPreordCMon$ and the reflection has stable units. Given a kernel $k\colon K\inclusion M$ in $\catPOCMon$, it is easy to check that the cokernel of $k$ is obtained by first taking the cokernel $q\colon M\to Q$ of $k$ in $\catPreordCMon$, and then composing $q$ with a reflection $\rho_Q\colon Q\to rQ$ of $P$ in $\catPOCMon$:
\[
\begin{tikzcd}[row sep = 1em]
    K\arrow[r, hookrightarrow, "k"]& M\arrow[r, "q"]&Q\arrow[r, "\rho_Q"] & rQ.
\end{tikzcd}
\]
One can then verify that a map $f\colon M\to N$ of partially ordered commutative monoids is a normal epimorphism if and only if it factors in $\catPreordCMon$ as $f=\comp g\rho$, with $g\colon M\to N'$ a normal epimorphism in $\catPreordCMon$, and $\rho\colon N'\to N$ a reflection of $N'$ in $\catPOCMon$. Since normal epimorphisms are stable in $\catPreordCMon$ (by \zcref{preord-cmon-is-normal}) and the reflector $\catPreordCMon\to\catPOCMon$ has stable units, we conclude that normal epimorphisms are stable in $\catPOCMon$.

The fact that $\catPOCMon$ is not regular follows from the same counterexample given in \zcref{preord-cmon-is-normal}, as all the involved preorders are actually partial orders.
\end{example}
\begin{example}
\label{p-set}
    The category $\catpSet$ of pointed sets is pointed, complete and cocomplete. A map of pointed sets $f\colon(X,x_0)\to (Y,y_0)$ is a normal epimorphism if and only if for each $y\in Y$, if $y\neq y_0$ then the fibre $f^{-1}(y)$ has cardinality exactly 1. 
    
    In general, the pullback of a normal epimorphism is not necessarily a normal epimorphism again. Consider for example the set $2=\{0,1\}$ (pointed in $0$) and the terminal object $1=\{0\}$; then, the unique map $p\colon 2\to 1$ is a normal epimorphism, but the pullback of $p$ along itself is a product projection $2\times 2\to 2$, which is clearly not a normal epimorphism according to the above characterisation.
    
    However, it is easy to check that the pullback of a normal epimorphism along a normal monomorphism (actually along any monomorphism) is again a normal epimorphism. This makes $\catpSet$ a \wpnormal{} category which is not \pnormal{}.
\end{example}
\begin{example}
\label{mp-pos}
   Consider the category $\mPOS$ of \emph{minimally pointed posets}, whose objects are triples $(X,\le,x_0)$, where $(X,\le)$ is a poset and $x_0$ is a minimal element in $(X,\le)$, and whose morphisms are monotone functions that preserve the distinguished minimal element. A function in $\mPOS$ is a normal epimorphism if and only if it is one in both $\catpSet$ and $\catPreordCMon$. Using arguments similar to those in \zcref{pocmon-is-normal,p-set}, one can show that $\mPOS$ is \wpnormal{} but neither \pnormal{} nor regular.
\end{example}
We conclude this section by summing up in a table the main examples and non-examples we have considered, each with its properties highlighted.
\begin{center}
\begin{tblr}{width=.8\textwidth,
colspec={r|c|c|c|c|c|c|c|c|},
row{4,6}={gray!10!white},
vline{1}={3-6}{},
hline{1}={2-9}{},
hline{2}={2-9}{},
cell{1}{2-9} = {cmd={\sideways}},
cell{2}{2-9} = {indaco!3!white}
}
&$\,\catAb\,$&$\,\catCMon$&$\,\catMon\,$&$\,\catRg\,$&$\,\catPreordCMon\,$&$\,\catPOCMon\,$&$\,\catpSet\,$&$\,\mPOS\,$\\
&\ref{abelian-implies}&\ref{cmon-is-normal}&\ref{mon-not-norm}&\ref{omega-cmon}&\ref{preord-cmon-is-normal}&\ref{pocmon-is-normal}&\ref{p-set}&\ref{mp-pos}\\
\hline
Regular? &Yes&Yes&Yes&Yes&No&No&Yes&No\\
\Wpnormal{}?&Yes&Yes&No&Yes&Yes&Yes&Yes&Yes\\
\Pnormal{}?&Yes&Yes&No&Yes&Yes&Yes&No&No\\
Normal?&Yes&No&No&No&No&No&No&No
\\
\hline
\end{tblr}
\end{center}
As was already mentioned,  the notion of \pnormality{} for pointed categories, while interesting in its own right, can be generalised to the broader setting of (non-pointed) categories equipped with a class of trivial objects. The next sections are devoted to this generalisation, beginning with the necessary preliminaries.
\section{Categories with a distinguished class of trivial objects}
\label{sec-triv-obj}
This section contains the necessary background on kernels and cokernels defined relative to a class of trivial objects. We omit most of the proofs as they concern well-known results. A more in-depth discussion of this topic can be found for example in \cite{GRANDIS13,GRANDIS92,MARKI13}.

Throughout this section, we fix a category $\cat C$ and a class $\tobj$ of objects in $\cat C$, which we identify with the full subcategory of $\cat C$ spanned by those objects.
\begin{definition}
\label{Z-def}
     We say that a morphism in $\cat C$ is \emph{\nptrivial{}} if it factors through an object in $\cat Z$.  Given any map $f\colon A\to B$ in $\cat C$, a \emph{\npkernel{}} of $f$ is given by an object $K$ and a map $k\colon K\to A$ such that $\comp kf$ is \nptrivial{} and for any other map $x\colon X\to A$ such that $\comp xf $ is \nptrivial{}, there exists a unique $x'\colon X\to K$ such that $\comp{x'}k=x$.
    By a slight abuse of notation, we may refer to the kernel $(K, k)$ simply as $K$ or $k$, when the distinction is clear from context.
    Of course, \emph{\npcokernel{}s} are defined dually.
    A \emph{\npexact{} sequence} is given by a pair of composable maps 
\(
\begin{tikzcd}[cramped, sep=1.5em]
A\arrow[r, "f"]&B\arrow[r,"g"]&C
\end{tikzcd}
\)
    such that $(A,f)$ is the \npkernel{} of $g$ and $(C,g)$ is the \npcokernel{} of $f$.
\end{definition}
\begin{remark}
\label{coker-of-ker}
 Given a \npkernel{} $(A,f)$, if $f\colon A\to B$ admits a \npcokernel{} $(C,g)$, then $(A,f)$ is the \npkernel{} of $g$.
\end{remark}
\begin{proposition}
\label{triv-cancel}
    Let $f$ be the underlying map of a \npkernel{}. Then $f$ is a monomorphism. Furthermore, if $g$ is a map such that $\comp gf$ is defined and \nptrivial{}, then $g$ itself is \nptrivial{}.
\end{proposition}
\begin{proof}
The first property is well known.
For the second property, suppose $f$ is the \npkernel{} of a map $u$. Since $\comp gf$ is \nptrivial{}, we can write $\comp gf=\comp st$ for some maps $s$ and $t$ such that $\cod s=\dom t$ is in $\tobj$. Observe that $\comp tu$ is \nptrivial{}, and hence $t$ factors uniquely as $t=\comp vf$ for some map $v$. It follows that $\mcomp{s,v,f}=\comp st=\comp gf$. As $f$ is a monomorphism, we obtain $g=\comp sv$, which shows that $g$ is \nptrivial{}.
\end{proof}
\begin{notation}
\label{norm-mono-notation}
    We call \emph{\npnormmono{}s} the   monomorphisms in $\cat C$ underlying \npkernel{}s. We denote by $\NMono$ the class of \npnormmono{}s in $\cat C$, and we graphically represent them using the same special arrow `$\nmonoto$' we used for (ordinary) normal monomorphisms.
    Similarly, we call \emph{\npnormepi{}s} the   epimorphisms underlying \npcokernel{}s. We denote by $\NEpi$ the class of \npnormepi{}s, and we use again  the special arrow `$\nepito$' to graphically represent them.
\end{notation}
\begin{proposition}
\label{stable-nm}
    Let $f\colon X\to Y$ be a morphism in $\cat C$ and let  $k\colon K\to X$ be its \npkernel{}. Given any morphism $x\colon X'\to X$, if the pullback of $k$ along $x$ exists, then it is the \npkernel{} of the composite $\comp xf$. In particular, \npnormmono{}s are stable under pullbacks.
\end{proposition}
\begin{remark}
\label{closed-under-retracts}
    The class of \nptrivial{} morphisms in $\cat C$ constitutes an \emph{ideal} of maps (\cite{EHRESMAN,GRANDIS92,MARKI13}). Replacing $\tobj$ with its closure under retracts (i.e.\ replacing $\tobj$ with the class $\tobj'$ of objects $X$ such that there exists a split monomorphism with domain $X$ and codomain in $\tobj$) does not change the corresponding ideal of trivial maps, nor the associated notions of kernel, cokernel and exactness. On the other hand, working with a class $\cat Z$ closed under retracts allows for cleaner statements of certain results. Accordingly, we may make the following assumption without loss of generality.
\end{remark}
From now on, and throughout this section, we assume that $\tobj$ is closed under retracts in $\cat C$.
\begin{proposition}
\label{monocoref}
     Given an object $B$ in $\cat C$ and a monomorphism $\cuni B\colon \corefl B\monoto B$, then $(\corefl B,\cuni B)$ is a coreflection of $B$ in $\tobj$ if and only if it is a \npkernel{} of $\id B$. 
    In this case, given maps 
    \(
    \begin{tikzcd}[cramped, sep=1.5em]
    K\arrow[r, "k"]&A\arrow[r,"f"]&B,
    \end{tikzcd}
    \)
    then $(K,k)$ is the \npkernel{} of $f$ if and only if there exists a (unique) map $h\colon K\to \corefl B$ such that the following square is a pullback.
    \begin{equation*}
    \begin{tikzcd}
    K\arrow[r, "h"]\arrow[d, "k"'] & \corefl B\arrow[d, "\cuni B"]
    \\
    A\arrow[r, "f"'] &B
    \end{tikzcd}
    \end{equation*}
    Finally, $\tobj$ is a mono-coreflective subcategory of $\cat C$ if and only if every identity morphism in $\cat C$ admits a \npkernel{}.
\end{proposition}

(In this paper, reflective and coreflective subcategories are always assumed to be closed under isomorphisms -- and therefore under retracts.)
\begin{proposition}
\label{eq-tker}
Given a \npnormmono{} $f\colon A\to B$ in $\cat C$, then the following are equivalent.
\begin{enumerate*}[(i)]
    \item The map $f$ is \nptrivial{};
    \item $A$ is in $\cat Z$;
    \item $(A,f)$ is the \npkernel{} of an isomorphism;
    \item $(A,f)$ is a coreflection of $B$ in $\cat Z$.
\end{enumerate*}
\end{proposition}
\begin{notation}
We denote by $\TKer$ the class of maps whose \npkernel{} exists and is \nptrivial{} (according to any of the equivalent statements in \zcref{eq-tker}), and we graphically represent a map with a \nptrivial{} \npkernel{} by using again the special arrow `\tkerto'.
\end{notation}
\begin{proposition}
\label{charact-tker}
     Suppose $f$ is a map in $\cat C$ admitting a \npkernel{}. Then the \npkernel{} of $f$ is \nptrivial{} if and only if, for every map $g$ such that $\comp gf$ is defined and \nptrivial{}, the map $g$ itself is \nptrivial{}.
\end{proposition}
\begin{proof}
    Let $k$ be the \npkernel{} of $f$. If $k$ is \nptrivial{} and $g$ is such that $\comp gf$ is defined and \nptrivial{}, then $g$ factors through $k$, which is a \nptrivial{} map, and so $g$ is itself \nptrivial{}. The converse is immediate.
\end{proof}
\begin{proposition}
\label{ne-ort-tk}
In $\cat C$, \npnormepi{}s satisfy the left lifting property with respect to maps with \nptrivial{} \npkernel{}.
\end{proposition}
\begin{proof}
    Consider the following commutative square,
    \begin{equation*}
    \begin{tikzcd}
    \cdot \arrow[r,"v"] \arrow[d,normepi, "f"'] &\cdot\arrow[d,tker, "g"]
    \\
    \cdot \arrow[r,"u"']&\cdot
    \end{tikzcd}
    \end{equation*}
    where $f$ is the \npcokernel{} of a map $a$ and and $g$ is a map whose  \npkernel{} $b$ is \nptrivial{}. Since $\mcomp{a,v,g}$ is \nptrivial{}, it follows from \zcref{charact-tker} that $\comp av$ is \nptrivial{} as well. Hence, there exists a unique map $d$ such that $v=\comp fd$, and by the epimorphicity of $f$, we further deduce that $\comp dg =u$.
\end{proof}
\section{Non-pointed \pnormal{} categories}
\label{non-p-normal}
Building on the preparatory material of the previous section, we are now ready to formalise the generalisation of \pnormal{} categories to the non-pointed setting.
\begin{definition}
\label{normal-def}
Let $\cat C$ be a category and $\cat Z$ a subcategory. We say that $\cat C$ is \emph{\pnormal{} with respect to $\tobj$}, or that $\cat C$ is \emph{\normal{}}, if the following properties hold.
\begin{enumerate}[start=0]
    \item $\tobj$ is mono-coreflective in $\cat C$.
    \item $\cat C$ has finite limits.
    \item $\cat C$ has \npcokernel{}s of \npkernel{}s.
    \item \npnormepi{}s are pullback-stable in $\cat C$.
\end{enumerate}
In what follows, when we say that $\cat C$ is a \normal{} category, it is understood that $\tobj$ is a subcategory of $\cat C$, and that $\cat C$ and $\tobj$ satisfy the conditions stated above.
\end{definition}
\begin{remark}
\label{kernel-in-preno}
Since (pointed) kernels are limits, in (pointed) \pnormal{} categories the existence of kernels follows  by finite completeness. In the non-pointed case the situation is less straightforward: the existence of \npkernel{}s of identity maps is not automatic, but it is equivalent to the mono-coreflectivity of $\tobj$, by \zcref{monocoref}. Once the existence of such kernels is ensured, the existence of every other \npkernel{} follows from finite completeness, again by \zcref{monocoref}.
\end{remark}

We can now state one of the main features of \pnormal{} categories in the non-pointed case.
\begin{proposition}
\label{norm-fs}
Let $\cat C$ be a \normal{} category. Then $\cat C$ admits the pair \facsys{} as a stable factorisation system.
\end{proposition}
\begin{proof}
    Clearly, $\NEpi$ and $\TKer$ are closed under compositions with isomorphisms. 
    More\-over, we have already verified in \zcref{ne-ort-tk} that every \npnormepi{} is orthogonal to every map with \nptrivial{} \npkernel{}. 
    Now, given any morphism $f$ in $\cat C$, consider the following diagram,
    \begin{equation*}
    \begin{tikzcd}
    \cdot\arrow[r, "k", normmono]\arrow[d,"{e'}"']&\cdot\arrow[r, "f"]\arrow[d, "e",normepi]&\cdot
    \\
    \cdot \arrow[r, "h"', normmono]&\cdot\arrow[ur, "m"', bend right]
    \end{tikzcd}
    \end{equation*}
    where $k$ is the \npkernel{} of $f$, $e$ is the \npcokernel{} of $k$ and $m$ is the unique map such that $f=\comp e m$, by the universal property of $e$. 
    The morphism $h$ is the \npkernel{} of $m$ and $e'$ is the unique map such that $\comp ke=\comp{e'}{h}$ by the universal property of $h$. The left square is a pullback, and therefore $e'$ is a \npnormepi{}. 
    Since $\comp{e'}h=\comp ke $ is \nptrivial{}, it follows from the dual of \zcref{triv-cancel}, that $h$ is itself \nptrivial{}, and thus $f$ factors as a \npnormepi{} followed by a map with a \nptrivial{} \npkernel{}, as desired.
\end{proof}
Before proceeding further, we pause to make a few general remarks on some features of the definition of non-pointed prenormal category and alternative approaches to similar ideas.
{
\begin{remark}
\label{ideals-2}
 In this remark we discuss the choice of working with a distinguished class of trivial objects rather than an ideal of morphisms (see again \cite{EHRESMAN,GRANDIS92,MARKI13}). Of course, 
\zcref{normal-def} could easily be reformulated using ideals of morphisms rather than trivial objects. However, some crucial results would not carry over. In particular, 
\zcref{norm-fs} 
no longer holds in the ideal-based approach: what fails is the dual of the cancellation property in \zcref{triv-cancel}. The following example illustrates this failure. (Notice that, while \zcref{norm-fs} is stated for finitely complete categories, the proof only requires the existence of pullbacks, which exist in the following example.)

Consider the category $\cat C$ having just one object $\ast$, whose endomorphisms are all positive integers, and with composition given by multiplication. This category has pullbacks. In fact, given any positive integers $m$ and $n$, let $d$ be their greatest common divisor, and write $m=ad$ and $n=bd$. Then it is easy to check that the following square is the pullback of $m$ along $n$.
\[
\begin{tikzcd}
    \ast \arrow[r, "a"] \arrow[d,"b"']& \ast\arrow[d,"n"] 
    \\
    \ast\arrow[r,"m"'] &\ast 
\end{tikzcd}
\]
\renewcommand{\tobj}{\cat N}Consider now the ideal $\tobj$ of arrows in $\cat C$ consisting of positive integers divisible by $10$. This ideal is obviously not \emph{closed}, meaning that there is no class of objects $\cat Z$ such that $\tobj$ is the class of $\cat Z$-trivial maps. Given any map $n$, its \npkernel{} and \npcokernel{} are as follows:
\[
\tobj\textnormal{-ker}(n)=\tobj\textnormal{-coker}(n)=2^{D(n,5)}\cdot5^{D(n,2)},
\]
with $D(n,k)=0$ if $k$ divides $n$ and $D(n,k)=1$ otherwise.
We notice then that the set of \npnormepi{}s is just $\{1,2,5,10\}$, and it coincides with the set of \npnormmono{}s. Since \npnormmono{}s are always pullback-stable (as one quickly checks), we readily conclude that \npnormepi{}s are pullback-stable in $\cat C$.

We show now that in this category, despite it having pullbacks, \npkernel{}s, \npcokernel{}s and stable \npnormepi{}s, we do not have a \facsys{}-factorisation. First of all, notice that if $f$ is a map which factors as $f=\comp q m$, with $q\in\NEpi$ and $m\in \TKer$ and $k$ is the \npkernel{} of $f$, then $q$ is the \npcokernel{} of $k$ (this is an easily proved general fact). Consider now the map $25$ and suppose it factors as $25=\comp q m$ with $q\in\NEpi$ and $m\in \TKer$. The \npkernel{} of 25 is $k=2$. Since $q$ must be the \npcokernel{} of $k$, we have $q=5$, and so $m$ must be 5 as well. However, the \npkernel{} of 5 is 2, which is not in $\tobj$, and so $m=5\notin\TKer$.
\end{remark}
}
\begin{remark}
    Different approaches exist in the literature that introduce non-pointed notions of normal epimorphisms (or normal monomorphisms) that can participate in (possibly stable) factorisation systems. In particular, in the very recent  work \cite{THOLEN25}, a `normal monomorphism' is defined -- in a category with enough (co)limits -- as a map $f\colon X \to Y$ whose pushout along the terminal map $X \to 1$ is also a pullback, thereby bypassing the necessity for a (reflective) subcategory $\tobj$ of trivial objects 
    altogether. Note however that if $\tobj$ is taken to be the subcategory of terminal objects, then \npnormmono{}s (in the sense of \ref{norm-mono-notation}) are normal monomorphisms in the sense of \cite{THOLEN25}, though the converse is not true in general (for instance, in $\catSet$, any function $1\to1+1$ provides a counterexample). Moreover, such $\tobj$ is not generally coreflective. These features highlight notable differences between \cite{THOLEN25} and the present paper. Nevertheless, the normal monomorphisms from \cite{THOLEN25}, and their dual notion of normal epimorphisms, can still be part of (not necessarily stable) factorisation systems, as established in \cite[Theorem~4.2]{THOLEN25} and examples in the same work.
\end{remark}
We conclude the section by highlighting a few properties of trivial objects in non-pointed \pnormal{} categories, most of which are trivial or invisible in the pointed case.
\newcommand{\srefi}{\cat C^{\to\tobj}}\newcommand{\srefin}{\cat{E}^{\to\cat A}}

In the definition of a \normal{} category $\cat C$, we require the subcategory $\tobj$ of trivial objects to be mono-coreflective, but not necessarily reflective, as the latter is not needed in general (see \zcref{trivial-ex}). However, $\tobj$ does satisfy a weaker but closely related condition. In the remainder of the section, we make this condition precise and explore its implications.
\begin{definition}
\label{subref-def}  
Given a category $\cat E$ and a subcategory $\cat A$, we say that $\cat A$ is \emph{subreflective} in $\cat E$ if the following holds: letting $\srefin$ denote the full subcategory of $\cat E$ consisting of those objects that admit a morphism into some object of $\cat A$, then every object in $\srefin$ admits a reflection into $\cat A$; in other words, $\cat A$ is a reflective subcategory of $\srefin$.
\end{definition}
\begin{proposition}
\label{norm-subref-triv}
    If $\cat C$ is a \normal{} category, then $\tobj$ is subreflective in $\cat C$.
\end{proposition}
\begin{proof}
It suffices to observe that for every object $X$ in $\srefi$, $\id X$ is a \npnormmono{}, and therefore admits a \npcokernel{}.
\end{proof}
\begin{remark}
\label{closure-property}
In finitely complete categories,  the notions of reflectivity and subreflectivity are related as follows: a subreflective subcategory of a finitely complete category is reflective if and only if it is closed under (finite) limits, if and only if it contains the terminal object. 

Moreover, if a subcategory $\tobj$ of a category $\cat C$ is both subreflective and mono-coreflective, then it satisfies several interesting structural properties:
\begin{enumerate}
        \item\label{cp1} $\tobj$ is closed under colimits and non-empty limits existing in $\cat C$.
        \item \label{cp2} If $\cat C$ admits colimits of a given shape, then maps from a colimit of that shape are \nptrivial{} if and only if their components are. 
        \item \label{cp3}If $\cat C$ admits limits of a given non-empty shape, then maps to limits of that shape are \nptrivial{} if and only if their components are.
        \item \label{cp4} If an object $X$ admits a reflection $(\refl X,\uni X)$, then $\uni X$ is an epimorphism (and therefore $(\refl X,\uni X)$ is the \npcokernel{} of $\id X$).
        \item \label{cp5} Every strong monomorphism in $\cat C$ has a \nptrivial{} \npkernel{}.
    \end{enumerate}
    \zcref[capfirst=true]{cp1,cp2,cp3,cp4} can be established essentially as in the reflective case, with minor adaptations, while \zcref[capfirst=true]{cp5} follows from \zcref{cp4}.
\end{remark}

\section{\Pnormality{} and factorisation systems}
\label{sec-normal-fs}
This section is devoted to establishing and proving the non-pointed analogue of \zcref{pnorm-fs}, providing a characterisation of \pnormal{} categories in terms of certain factorisation systems.
{
\newcommand{\lcl}{\mathscr{E}}
\newcommand{\rcl}{\mathscr{M}}
\begin{proposition}
\label{characterisation-fs}
    Let $\cat C$ be a category with finite limits, and let $\cat Z$ be a mono-coreflective subcategory of $\cat C$. Then the following are equivalent.
    \begin{enumerate}[(i)]
        \item \label{pnfs1} $\cat C$ is \normal{}.
        \item\label{pnfs2} $\cat C$ admits \facsys{} as a stable factorisation system.
    \end{enumerate}
    If $\cat Z$ is \sreflective{} (see \zcref{subref-def}), then the following are also equivalent to \zcref{pnfs1,pnfs2}.
    \begin{enumerate}[(i)]
    \setcounter{enumi}{2}
        \item\label{pnfs3} $\cat C$ admits a stable factorisation system of the form $\bigl(\NEpi,\rcl\bigr)$ for some class $\rcl$ of arrows in $\cat C$.
        \item\label{pnfs4} $\cat C$ admits a stable factorisation system of the form $\bigl(\lcl,\TKer\bigr)$ for some class $\lcl$ of arrows in $\cat C$.
    \end{enumerate}
\end{proposition}
\begin{proof}
    We have already proved that \ref{pnfs1} $\!\implies\!$ \ref{pnfs2} in \zcref{norm-fs}, and obviously \ref{pnfs2}$\implies$\ref{pnfs3} and \ref{pnfs2}$\implies$\ref{pnfs4}. We now prove the remaining implications. Let us begin by showing that \ref{pnfs2}$\implies$\ref{pnfs1}.
    \begin{itemize}
        \item Suppose \ref{pnfs2} holds. To show that $\cat C$ is \normal{} we just need to prove that $\cat C$ admits \npcokernel{}s of \npkernel{}s. Let $(A,f)$ be the \npkernel{} of a map $g\colon B \to C$. Consider the factorisation $g=\comp em$, with $e\in\NEpi$ and $m\in\TKer$. It follows from \zcref{charact-tker} that $(A,f)$ is also the \npkernel{} of $e$, and therefore  $e$ is the \npcokernel{} of $f$ by the dual of \zcref{coker-of-ker}. \end{itemize}
    We now proceed to prove that \ref{pnfs3}$\implies$ \ref{pnfs2} and that \ref{pnfs4}$\implies$\ref{pnfs2} under the assumption that $\tobj$ is \sreflective{}.
    \begin{itemize}
        \item Suppose \ref{pnfs3} holds. We show that $\rcl=\TKer$. We know from \zcref{ne-ort-tk} that \npnormepi{}s satisfy the left lifting property with respect to maps with a \nptrivial{} \npkernel{}. It then follows from \zcref{fs-property} in \zcref{fs-properties} that $\TKer\subseteq\rcl$. It remains to prove the reverse inclusion. Let $f\colon A\to B$ be a map in $\rcl$, and let $(K,k)$ be its \npkernel{} (whose existence is guaranteed by \zcref{kernel-in-preno}). Since the composite $\comp k f\colon K\to B$ is \nptrivial{}, \sreflectivity{} of $\tobj$ ensures that $K$ admits a reflection $(\refl K,\uni K)$ in $\cat Z$. We thus obtain the following commutative square for some map $\refl K \to B$.
        \begin{equation*}
        \begin{tikzcd}
            K\arrow[r,"k"]\arrow[d,"\uni K"'] &A\arrow[d,"f"]
            \\
            \refl K\arrow[r] & B
        \end{tikzcd}
        \end{equation*}
        By \zcref{cp4} in \zcref{closure-property}, $\uni K$ is a \npnormepi{}, and hence it satisfies the left lifting property with respect to maps in $\rcl$. We then obtain a map $d\colon \refl K\to A$ such that $k=\comp{\uni K}d$, and so $k$ is trivial, as desired.
        \item Suppose now that \ref{pnfs4} holds. We show that $\lcl=\NEpi$. Arguing as in the previous point, we have that $\NEpi\subseteq\lcl$. We need to prove the reverse inclusion. We proceed in steps.
        \begin{enumerate}[a)]
        \item\label{letter-a} First let us show that  if $f\colon A\to B$ lies in $\lcl$, and $g\colon B\to C$ is such that the composite $\comp fg$ is \nptrivial{}, then  $g$ is itself \nptrivial{}. Consider a coreflection $(\corefl C,\cuni C)$ of $C$ in $\cat Z$, so that we have the following factorisation for some map $A\to\corefl C$.
        \begin{equation*}
        \begin{tikzcd}
            A\arrow[r]\arrow[d,"f"'] &\corefl C\arrow[d,"\cuni C"]
            \\
            B\arrow[r,"g"']& C 
        \end{tikzcd}
        \end{equation*}
        Since $\cuni C$ has a \nptrivial{} \npkernel{} (as it is a \npnormmono{} -- see \zcref{sec-triv-obj}) and $f\in\lcl$, there exists a map $d\colon B\to\corefl C$ such that $g=\comp d{\cuni C}$. Hence, $g$ is \nptrivial{}.
        
        \item\label{letter-b} Now let $f\colon A\to B$ be a map in $\lcl$ and let $(K,k)$ be its \npkernel{}. We claim that if $g\colon A\to C$ is an arrow such that $\comp kg$ is \nptrivial{}, then there exists a map $g'$ such that $g=\comp f{g'}$; in other words, $(B,f)$ satisfies the universal property of a \npcokernel{} of $k$, but without uniqueness.
        
        Consider the product $B\times C$, with projections $p_B$ and $p_C$,
and let $h\colon A\to B\times C$ be the map with components $f$ and $g$, as shown in the following diagram.
        \[
        \begin{tikzcd}[row sep =1 em]
        & B
        \\
        A\arrow[ur,"f"]\arrow[dr, "g"']\arrow[rr, "h" description] & &B\times C\arrow[ul, "p_B"']\arrow[dl, "p_C"]  \\
        &C
        \end{tikzcd}
        \]
        We can factor $h$ as $h=\comp e m$, for some $e\colon A\to M$ in $\lcl$ and $m\colon M\to B\times C$ in $\TKer$.  Since $f$ and $e$ are both in $\lcl$, we can apply \zcref{fs-property-cancel} of \zcref{fs-properties} to the equality $f=\mcomp{e,m,p_B}$ and immediately conclude that $\phi\coloneqq\comp m{p_B}\in\lcl$.

        Next, let $(\corefl B,\cuni B)$ be a coreflection of $B$ in $\tobj$. Since $(K,k)$ is the \npkernel{} of $f$, there exists a (unique) map  $w\colon K\to\corefl B$ such that the square of sides $f$, $\cuni B$, $w$ and $k$  is a pullback (see \zcref{monocoref}). We can then construct the diagram below, where $(L,l,y)$ is the pullback of $\cuni B$ and $\comp m{p_B}$, and $x$ is the unique map such that $\comp x y=w$ and $\comp xl =\comp ke$.
        \begin{equation*}
        \begin{tikzcd}[column sep= 4em]
        K\arrow[r,"x"]\arrow[d,"k"]\arrow[rr, bend left, "w"]& L\arrow[r,"y"]\arrow[d,"l"] & \corefl B\arrow[d,"\cuni B"]
        \\
        A\arrow[r, "e"']\arrow[rr,"f"', bend right ] & M\arrow[r,"{\comp m {p_B}}"'] & B
        \end{tikzcd}
        \end{equation*}
        First, since the right-hand square is a pullback, it follows from \zcref{monocoref} that $(L,l)$ is a \npkernel{} of $\comp m{p_B}$. Second, since the outer rectangle is a pullback (as mentioned above), we deduce the left-hand square is also a pullback. Hence, $x\in\lcl$ as it is a pullback of $e\in\lcl$. 

        Consider now the morphism
        \(
            \mcomp{x,l,m}=\mcomp{k,e,m}=\comp kh
        \) into the product $B\times C$. The components of this morphism are $\comp kf$ and $\comp kg$, both \nptrivial{} by hypothesis.  By \sreflectivity{} of $\tobj$  it follows that $\comp hk$, and hence $\mcomp{x,l,m}$, is \nptrivial{} (see \zcref{closure-property}, \zcref{cp3}). 
        We can then apply \zcref{letter-a} above and \zcref{charact-tker}  to cancel $m$ on the left and $x$ on the right, and conclude that $l$ itself is \nptrivial{}. 
        
        As previously observed, $(L,l)$ is the \npkernel{} of $\phi=\comp m{p_B}$, and so we have shown that $\phi\in\TKer$. We also established earlier that $\phi\in\lcl$ as well, and hence $\phi$ is an isomorphism (use \zcref{fs-properties}). We can therefore define the map $g'=\mcomp{\phi^{-1},m,p_C}$, which satisfies $\comp f{g'}=g$.
         \item\label{letter-c} Finally, we show that any map in $\lcl$ is an epimorphism. By \zcref{closure-property} (\zcref{cp5}), every strong monomorphism has a \nptrivial{} \npkernel{}. It follows that every map in $\lcl$ satisfies the left lifting property with respect to strong monomorphisms, which implies -- in the presence of equalisers -- that they are epimorphisms.
        \end{enumerate}
       Combining \zcref{letter-b,letter-c} we obtain that any map in $\lcl$ is the \npcokernel{} of its \npkernel{}, as required.\qedhere
    \end{itemize}
\end{proof}
}
\section{General properties of (non-pointed) \pnormal{} categories}
\label{sec-properties}
In this section, we extend to the general non-pointed case the properties previously established for \pnormal{} categories in \zcref{pointed-case}, this time providing full proofs. We begin with a result generalising some well-known properties of pullbacks holding for regular categories.
\begin{lemma}
\label{pullback-lemmas}
   Let $\cat C$ be a category with pullbacks.
    \begin{enumerate}
        \item\label{pb-po-part} Consider the following pullback diagram in $\cat C$.
        \begin{equation}
    \label{pb-po-square}
    \begin{tikzcd}
    \cdot\arrow[r, "q",epi]\arrow[d, "g"',epi]&\cdot\arrow[d, "f"]
    \\
    \cdot\arrow[r,regepi, "p"']&\cdot
    \end{tikzcd}
    \end{equation}
    If $p$ is a regular epimorphism, $g$ is a pullback-stable epimorphism and $q$ is an epimorphism, then \ref{pb-po-square} is also a pushout square.
    \item\label{pf-pb-part} Consider then the following commutative diagram in $\cat C$.  
    \begin{equation}
    \label{pf-pb-diag}
\begin{tikzcd}
    A\arrow[r,"f"]\arrow[d, "a"] & B\arrow[r, "g"]\arrow[d,"b"] & C\arrow[d, "c"]
    \\
    A'\arrow[r, "{f'}"', regepi] & B'\arrow[r, "{g'}"'] &C'
\end{tikzcd}
\end{equation}
    If $f'$ is a pullback-stable regular epimorphism and both the outer rectangle and the left-hand square are pullbacks, then the right-hand square is also a pullback.
\item \label{B-K} (Generalised Barr-Kock theorem) Consider the following commutative diagram where $f$ is a pullback-stable regular epimorphism, $(r_0,r_1)$ is the kernel pair of $f$ and $(s_0,s_1)$ is the kernel pair of $g$.
\begin{equation}
\label{B-K-diag}
   \begin{tikzcd}
        \cdot \arrow[r, yshift=.3em, "r_0"]\arrow[r,yshift=-.3em, "r_1"']\arrow[d]& \cdot\arrow[r,regepi, "f"] \arrow[d] & \cdot\arrow[d]
        \\
        \cdot \arrow[r, yshift=.3em, "s_0"]\arrow[r,yshift=-.3em, "s_1"']& \cdot\arrow["g"',r] & \cdot
    \end{tikzcd}
\end{equation}
If either of the left-hand squares is a pullback, then the right-hand square is a pullback.
    \end{enumerate}
\end{lemma}
\begin{proof}
\zcref[capfirst=true]{pb-po-part} may be proved essentially as in the regular case. \zcref[capfirst=true]{B-K} may also be proved as in the regular case, using \zcref{pf-pb-part} (see, for example, \cite[Theorem~1.16]{GRAN21}).
For \zcref{pf-pb-part}, we adapt the proof given in  \cite{GRAN21} for the analogous result. Consider the following commutative diagram.
\begin{equation*}
\begin{tikzcd}[column sep =2.5 em, row sep = 2 em]
    A &[-1em]&[.5em]B &[-1em]&[.5em] C
    \\
    & P&&Q&& C
    \\
    A' && B' && C'
    \arrow[from =1-1, to=1-3, "f",regepi]
    \arrow[from=1-3, to =1-5, "g"]
    \arrow[from=2-2, to=2-4, "p" near start,regepi]
    \arrow[from=2-4,to=2-6, "q" near start]
    \arrow[from=3-1,to=3-3, "{f'}"', regepi]
    \arrow[from=3-3,to=3-5, "{g'}"']
\arrow[from=1-1,to=3-1, "a"' near end]
    \arrow[from=1-3,to=3-3, crossing over, "b"' near end]
    \arrow[from=1-5, to=3-5, crossing over, "c"' near end]
\arrow[from=1-1,to=2-2, "\alpha"]
    \arrow[from=2-2,to=3-1, "{p'}"{xshift=.2em, yshift=.5em}]
    \arrow[from=1-3,to=2-4, "\beta"]
    \arrow[from=2-4,to=3-3,"{q'}"{xshift=.2em, yshift=.5em}]
    \arrow[from=1-5,to=2-6, equal]
    \arrow[from=2-6,to=3-5, "{c}"{xshift=.4em, yshift=.3em}]
\end{tikzcd}
\end{equation*}
Here $(Q,q,q')$ is the pullback of $g'$ and $c$, $(P,p,p')$ is the pullback of $f'$ and $q'$, while $\alpha$ and $\beta$ are the induced maps. Note that $p$ and $f$ are regular epimorphisms, as they arise as pullbacks of $f'$, while $\alpha$ is an isomorphism, since the outer rectangle is a pullback. Therefore, the following square
\begin{equation}
\label{some-square-pb-transf}
\begin{tikzcd}
    A\arrow[r, "f", regepi]\arrow[d, "\alpha"'{xshift=-.4em}, "\iso"{description}]& B\arrow[d, "\beta"]
    \\
    P\arrow[r, "p"', regepi] & Q
\end{tikzcd}
\end{equation} 
is a pullback square where $p$ and $f$ are regular epimorphisms and $\alpha$ is an isomorphism. We can then apply \zcref{pb-po-part} and deduce that \ref{some-square-pb-transf} is also a pushout, which proves that  $\beta$ is an isomorphism.
\end{proof}
We can now establish some pullback properties for non-pointed \pnormal{} category, most following from \zcref{pullback-lemmas}. 
\begin{proposition}
\label{pb-po-prop}
    In a \normal{} category $\cat C$, consider the following pullback diagram.
    \begin{equation}
    \zcsetup{reftype=diagram} \label{becomes-a-pushout}
    \begin{tikzcd}
        X'\arrow[r, "f'"]\arrow[d,"x"']&Y'\arrow[d, "y"]
        \\
        X\arrow[r, "f"'] &Y
    \end{tikzcd}
    \end{equation}
    \zcref[capfirst=true]{becomes-a-pushout} is also a pushout whenever one of the following conditions holds.
    \begin{enumerate*}[(a)]
        \item\label{pb-po-corefl} $f$ is a \npnormepi{} and the $\tobj$-coreflection of $y$ is an isomorphism.
        \item\label{pb-po-nm}$f$ is a \npnormepi{} and $y$ is a \npnormmono{}.
        \item\label{pb-po-ne} $f$ is a regular epimorphism, $x$ is a \npnormepi{} and $f'$ is an epimorphism.
    \end{enumerate*}
\end{proposition}
{
\renewcommand{\corefl}[1]{S{#1}}
\begin{proof}
\zcref[capfirst=true]{pb-po-ne} follows from \zcref{pullback-lemmas}, \zcref{pb-po-part}. \zcref[capfirst=true]{pb-po-nm} is a special case of \zcref{pb-po-corefl}, as the $\tobj$-coreflection of a \npnormmono{} is always an isomorphism. To prove \zcref{pb-po-corefl}, let $\cunn\colon \corefl{}\tto\id{\cat C}$ be a coreflection of $\tobj$ in $\cat C$, and consider the following diagram, where the left square is the pullback of $f'$ along $\cuni {Y'}$.
\[
\begin{tikzcd}
    K\arrow[r, "k",normmono] \arrow[d, "f''", normepi]&X'\arrow[r, "x"]\arrow[d,"f'",normepi]&X\arrow[d, "f", normepi]
    \\
    \corefl{Y'}\arrow[r,"\cuni{Y'}"',normmono]& Y'\arrow[r,"y"']&Y
\end{tikzcd}
\]
By \zcref{ex-seq-pb-po}, the left-hand square is a pushout. Since the $\tobj$-coreflection of $y$ is an isomorphism, it follows that $\mcomp{\cuni{Y'},y}=\mcomp{\corefl y,\cuni Y}\iso\cuni Y$ is a $\tobj$-coreflection of $Y$. Therefore, again by \zcref{ex-seq-pb-po}, the outer rectangle in the above diagram is a pushout. We thus conclude that \zcref{becomes-a-pushout} is a pushout.
\end{proof}}
\begin{proposition}
\label{normal-pullback-lemmas}
    In a \normal{} category where \npnormepi{}s are also regular epimorphisms, the following properties hold.
    \begin{enumerate}
        \item  In any diagram as in \ref{pf-pb-diag}, if $f'$ is a \npnormepi{} and both the outer rectangle and the left-hand square are pullbacks, it follows that the right-hand square is also a pullback.
        \item In any diagram as in \ref{B-K-diag}, if $f$ is a \npnormepi{} and either of left-hand squares is a pullback, it follows that the right-hand square is also a pullback.
    \end{enumerate}
\end{proposition}
We now resume our list of properties.
\begin{proposition}
\label{lemma-pb-ex-seq}
    Let $\cat C$ be a \normal{} category, and consider the following commutative diagram where $f$ is the \npkernel{} of $g$.
    \begin{equation*}
    \begin{tikzcd}
        A'\arrow[r,"{f'}"]\arrow[d,"a"']&B'\arrow[r,"{g'}"]\arrow[d,"b"']&C'\arrow[d,"c"']
        \\
        A\arrow[r, "f"',
        normmono]& B\arrow[r,"g"']& C
    \end{tikzcd}
    \end{equation*}
    \begin{enumerate}
    \item\label{pb-e-1} If the left square is a pullback and $g'$ is the \npcokernel{} of $f'$, then $c\in\pTKer$. 
    \item \label{pb-e-2}Conversely, if $c\in\pTKer$ and $f'$ is the \npkernel{} of $g'$, then the left square is a pullback.
    \end{enumerate}
\end{proposition}
    \begin{proof}
    To prove \zcref{pb-e-1}, notice that by \zcref{stable-nm} we have that $f'$ is the \npkernel{} of the map $d=\comp bg=\comp {g'}c$. Since $g'$ is the \npcokernel{} of $f'$, by the construction provided in the proof of \zcref{norm-fs}, we find that $\comp{g'}c$ is the \facsys{}-factorisation of $d$, and so $c\in\TKer$.

    The partial converse in \zcref{pb-e-2} does not rely on \normality. It readily follows from the universal property of \npkernel{}s and the characterisation of maps in $\TKer$ given in \zcref{charact-tker}.
    \end{proof}
    \begin{lemma}
    \label{product-norm-epi}
        In a \normal{} category $\cat C$, if $f\colon A\nepito B$ and $f'\colon A'\nepito B'$ are \npnormepi{}, then $f\times f'\colon A\times A'\to B\times B'$ is also a \npnormepi{}.
    \end{lemma}
\begin{proposition}
\label{slice-is-normal}
    Let $\cat C$ be a \normal{} category. If $C$ is an object in $\cat C$, then the slice category $\slice CC$ is a a $(\slice Z C)$-\pnormal{} category. If $\cat B$ is any category, then the functor category $\Fun BC$ is a $\Fun BZ$-\pnormal{} category.
\end{proposition}
\begin{proof}
See \zcref{ex-slices,ex-fun-cat}.
\end{proof}
In the final part of this section, we highlight some relationships between the notions introduced in this paper and the one in \cite{NORMAL}. The main difference between  \pnormal{} categories and normal categories as defined in \cite{NORMAL} is that, in the former, we do not ask for any relation between normal epimorphisms and regular epimorphisms. This distinction becomes particularly relevant if one wants to work in a non-pointed setting, as highlighted by the following propositions.
\begin{proposition}
\label{subterminal}
    Let $\cat C$ be a finitely complete category and $\tobj$ a class of trivial objects (closed under retracts). If $\RegEpi\subseteq\NEpi$, then every object in $Z$ in $\tobj$ is subterminal.
\end{proposition}
\begin{proof}
In the presence of kernel pairs, maps having the right lifting property with respect to split epimorphisms are monomorphisms. Therefore, if $\RegEpi\subseteq\NEpi$, then by \zcref{ne-ort-tk} we have that $\Mono\supseteq \TKer$. In particular, for any object $Z$ in $\tobj$ the terminal map $Z\to1$ is a monomorphism, since it has a \nptrivial{} \npkernel{} -- namely $(Z,\id Z)$. 
\end{proof}
\begin{proposition}
Let $\cat C$ be a pointed category, and suppose that $\cat C$ is \normal{} for some class $\tobj$ of trivial objects. The following conditions are equivalent:
\begin{enumerate*}[(i)]
    \item \label{pn-vs-n-1} $\TKer\subseteq \Mono$,
    \item \label{pn-vs-n-2} $\NEpi\supseteq\StrEpi$,
    \item \label{pn-vs-n-3} $\NEpi\supseteq\RegEpi$,
    \item \label{pn-vs-n-4} $\tobj$ is the class of zero objects of $\cat C$ and $\cat C$ is normal (in the sense of \cite{NORMAL}).
\end{enumerate*}
\end{proposition}
\begin{proof}

    The only non-obvious implication is \ref{pn-vs-n-3}$\implies$\ref{pn-vs-n-4}. Suppose 
    that \ref{pn-vs-n-3} holds. By \zcref{subterminal}, it follows that $\tobj$ is the class of zero objects (subterminal objects are obviously zero objects in a pointed category), and 
    that $\NEpi=\pNEpi=\RegEpi$. Therefore, $\cat C$ is a pointed, finitely complete category where regular epimorphisms are stable and coincide with normal epimorphisms, i.e.\ $\cat C$ is normal.
\end{proof}
\section{Exact sequences in \pnormal{} categories}
In this section, we study the properties of exact sequences in (non-pointed) \pnormal{} categories. We also consider functors that preserve such sequences along with finite limits, and show that they preserve much of the extra structure of \pnormal{} categories.
\label{sec-ex-seq}
\begin{proposition}
\label{ex-seq-pb-po}
Let $\cat C$ be a \normal{} category, and let \(
\begin{tikzcd}[cramped, column sep =2em]
    A\arrow[r, "f",normmono]&B\arrow[r, "g",normepi]&C
\end{tikzcd}
\)
be a \npexact{} sequence in $\cat C$. Then the composite $\comp fg$ factors through an object $Z$, i.e.\ there exist morphisms $\unin\colon A\to Z$ and $\cunn\colon Z\to C$ such that the following diagram commutes. 
\begin{equation}
\zcsetup{reftype=diagram} \label{exact-sequence-square}
\begin{tikzcd}
    A\arrow[r, "f"]\arrow[d, "\unin"']& B\arrow[d, "g"]
    \\
    Z\arrow[r, "\cunn"'] & C
\end{tikzcd}
\end{equation}
For any such factorisation, the following conditions are equivalent:
\begin{enumerate*}[(a)]
    \item\zcsetup{reftype=condition} \label{ex-a}\zcref{exact-sequence-square} is both a pullback and a pushout; 
    \item \zcsetup{reftype=condition} \label{ex-b}$(Z,\cunn)$ is a $\tobj$-coreflection of $C$; 
    \item \zcsetup{reftype=condition} \label{ex-c}$(Z,\unin)$ is a $\tobj$-reflection of $A$. 
\end{enumerate*}
\end{proposition}
\begin{proof}
    By \zcref{monocoref}, we may  consider a $\tobj$-coreflection $\cunn\colon Z\to C$ of $C$, and obtain a map $\unin\colon A\to Z$ such that \zcref{exact-sequence-square} is a pullback. Since $g$ is a \npnormepi{} and \zcref{exact-sequence-square} is a pullback, it follows that $\unin$ is also a \npnormepi{}. By the dual of \zcref{eq-tker}, $\unin$ is a $\tobj$-epi-reflection of $A$, and by the dual of \zcref{monocoref}, \zcref{exact-sequence-square} is a pushout.

    Now, given any diagram of the form of \ref{exact-sequence-square} satisfying either \ref{ex-b} or \ref{ex-c}, it must be isomorphic to the one constructed in the previous paragraph. This shows that \zcref{ex-b,ex-c} are equivalent, and that each of them implies \ref{ex-a}. 
    
    Conversely, given a diagram like \ref{exact-sequence-square} satisfying \zcref{ex-a}, then $\unin$ is a \nptrivial{} \npnormepi{}, and thus \ref{ex-b} holds by the dual of \zcref{eq-tker}.
\end{proof}
\begin{proposition}
\label{stability-of-ex-seq}
In a \normal{} category $\cat C$:
\begin{enumerate}
    \item\label{pr-ex-seq}binary products of \npexact{} sequences are \npexact{};
    \item\label{pb-ex-seq}  the pullback of a \npexact{} sequence along a morphism $c$ is \npexact{} if and only if $c\in\TKer$.
\end{enumerate} 
\end{proposition}
\begin{proof}
    Given two \npexact{} sequences
    \[
    \begin{tikzcd}
        A\arrow[r, "f",normmono]& B\arrow[r, "g",normepi] & C &\text{and}&  A'\arrow[r, "f'",normmono] &B' \arrow[r, "g'",normepi]&C',
    \end{tikzcd}
    \]
    we know that $g\times g'$ is a \npnormepi{} by \zcref{product-norm-epi}. The fact that $f\times f'$ is the \npkernel{} of $g\times g'$  follows from the observation that the product of mono-coreflections is a mono-coreflection of the product, together with \zcref{monocoref}.

    Consider the following diagram,
    \[
    \begin{tikzcd}
     A'\arrow[r,"{f'}"]\arrow[d,"a"']&B'\arrow[r,"{g'}"] \arrow[d,"b"']&C'\arrow[d,"c"']
        \\
        A\arrow[r, "f"', normmono]& B\arrow[r,"g"',normepi]& C
    \end{tikzcd}
    \]
    with both squares being pullbacks and the lower row \npexact{}. In general, $g'\in\NEpi$ by \normality{}, and $f'$ is the \npkernel{} of $\comp bg=\comp{g'}c$ by \zcref{stable-nm}. Moreover, if $c\in\TKer$ one readily verifies that $f'$ is also the \npcokernel{} of $g'$.  The converse implication follows from \zcref{lemma-pb-ex-seq}.
\end{proof}
{
\renewcommand{\tobj}{\cat U}
\newcommand{\dobj}{\cat V}
\newcommand{\dbasecategory}{\cat D}
\newcommand{\di}{\renewcommand{\tobj}{\dobj}\renewcommand{\basecategory}{\dbasecategory}}
\begin{proposition}
    Let $\cat C$ and $\cat D$ be categories, and let $\tobj$ and $\dobj$ be subcategories of $\cat C$ and $\cat D$ respectively. Suppose $\cat C$ is \normal{} and $\cat D$ is {\di\normal{}}. If $F\colon \cat C\to\cat D$ is a functor preserving finite limits and sending \npexact{} sequences to {\di\npexact{}} sequences, then $F$ preserves: 
    \begin{enumerate}
        \item\label{ex-funct1} trivial objects and trivial maps ($F(\tobj)\subseteq\dobj$ and if $f$ is \nptrivial{}, then $Ff$ is {\di\nptrivial{}});
        \item\label{ex-funct2} coreflections into trivial objects (if $(\corefl X,\cuni X)$ is a coreflection of an object $X$ of $\cat C$ in $\tobj$, then $(F\corefl X,F\cuni X)$ is a coreflection of $FX$ in $\dobj$);
        \item\label{ex-funct3} kernels (if $(K,k)$ is a \npkernel{} of $f$ in $\cat C$, then $(FK, Fk)$ is a {\di\npkernel{}} of $Ff$);
        \item \label{ex-funct4}the normal-epi-trivial-kernel factorisations ($F(\TKer)\subseteq{\di\TKer}$ and $F(\NEpi)\subseteq{\di\NEpi}$).
    \end{enumerate}
\end{proposition}
\begin{proof} 
\zcref[capfirst=true]{ex-funct1,ex-funct2} follow from the fact that $F$ sends \npexact{} sequences to {\di\npexact{}} sequences, together with the following observations:\begin{itemize}
    \item an object $X$ in $\cat C$ (respectively, in $\cat D$) lies in $\tobj $ (respectively, in $\dobj$) if and only if the sequence
    \(
    \begin{tikzcd}[cramped, sep=1.2em]
    X\arrow[r, equal]&X\arrow[r,equal]&X
    \end{tikzcd}
    \) 
    is \npexact{} (respectively, {\di\npexact{}});
    \item for an object $X$ in $\cat C$ (respectively, in $\cat D$), we have that $(\corefl X,\cuni X)$ is a $\tobj$-coreflection (respectively, a $\dobj$-coreflection) if and only if the sequence 
    \(
    \begin{tikzcd}[cramped, sep=1.7em]
    \corefl X\arrow[r, "\cuni X"]&X\arrow[r,equal]&X
    \end{tikzcd}
    \)
    is \npexact{} (respectively, {\di\npexact{}}).
\end{itemize}
\zcref[capfirst=true]{ex-funct3} follows immediately from the fact that $F$ preserves finite limits, from \zcref{ex-funct2} above and from the description of kernels in \zcref{monocoref}.

To prove that $F(\TKer)\subseteq{\di\TKer}$ just use \zcref{ex-funct3,ex-funct1}. To prove that $F(\NEpi)\subseteq{\di\NEpi}$, just notice that any \npnormepi{} $f$ is part of a \npexact{} sequence
\(
\begin{tikzcd}[cramped, sep=1.5em]
\cdot\arrow[r, "k "]&\cdot\arrow[r,"f"]&\cdot,
\end{tikzcd}
\)
where $k$ is the \npkernel{} of $f$.
\end{proof}
}
\section{A weaker notion: \wpnormal{} categories}
\label{weakly-normal}
In this section we want to explore a relaxation of the axioms of \pnormality{},  which we refer to as \emph{\wpnormality{}}. \zcref[capfirst=true]{wnorm-3} in \zcref{wnorm} below was first studied -- even in the non-pointed context -- by Grandis  (\cite{GRANDIS13,GRANDIS92}), albeit with different hypotheses, purposes and terminology. Closely related conditions in the pointed context were also investigated in \cite{TIM}. Besides the examples found in \zcref{pointed-case,Examples}, further instances of \wpnormal{} categories appear in \cite{GRANDIS13}, and a more recent one in \cite{CONNES}.

\begin{definition}
\label{wnorm}
Let $\cat C$ be a category and let $\tobj$ be a subcategory of $\cat C$. We say that $\cat C$ is \emph{\wnormal{}} if the following properties hold.
\begin{enumerate}[start=0]
    \item $\tobj$ is mono-coreflective in $\cat C$;
    \item\zcsetup{reftype=condition} \label{wnorm-1} $\cat C$ admits pullbacks along \npnormmono{}s;
    \item\zcsetup{reftype=condition} \label{wnorm-2}  $\cat C$ admits \npcokernel{}s of \npkernel{}s;
    \item \zcsetup{reftype=condition} \label{wnorm-3}the pullback of a \npnormepi{} along a \npnormmono{} is a \npnormepi{}.
\end{enumerate}
\end{definition}
In his works \cite{GRANDIS13,GRANDIS92}, Grandis did not relate the partial stability of normal epimorphisms to the factorisation system we have considered in this paper. Nevertheless, the characterisation given in \zcref{characterisation-fs} still holds for \wpnormal{} categories, though in the following weaker form.
{
\newcommand{\lcl}{\mathscr{E}}
\newcommand{\rcl}{\mathscr{M}}
\begin{proposition}
    Let $\cat C$ be a category and let $\cat Z$ be a mono-coreflective subcategory of $\cat C$. Suppose $\cat C$ admits pullbacks along \npnormmono{}s. Then the following are equivalent.
    \begin{enumerate}
        \item\label{wpnfs1} $\cat C$ is \wnormal{}.
        \item\label{wpnfs2} $\cat C$ admits \facsys{} as a factorisation system, and the class $\NEpi$ is stable under pullbacks along \npnormmono{}s.
    \end{enumerate}
    If $\cat C$ admits binary products and equalisers and $\cat Z$ is \sreflective{}, then the following are also equivalent to \zcref{wpnfs1,wpnfs2}.
    \begin{enumerate}
    \setcounter{enumi}{2}
        \item\label{wpnfs3} $\cat C$ admits a  factorisation system of the form $\bigl(\NEpi,\rcl\bigr)$ for some class $\rcl$ of arrows in $\cat C$, and the class $\NEpi$ is stable under pullbacks along \npnormmono{}s.
        \item\label{wpnfs4} $\cat C$ admits a  factorisation system of the form $\bigl(\lcl,\TKer\bigr)$ for some class $\lcl$ of arrows in $\cat C$ stable under pullbacks along \npnormmono{}s.
    \end{enumerate}
\end{proposition}
\begin{proof}
    The proof of \zcref{characterisation-fs} essentially works without any changes for this proposition as well.
\end{proof}
}

Some other properties we proved for \normal{} categories (in particular those in Propositions~\ref{norm-subref-triv}, \ref{pb-po-prop}.\ref{pb-po-nm}, \ref{lemma-pb-ex-seq}, \ref{slice-is-normal}, \ref{ex-seq-pb-po}) extend to \wnormal{} categories with no changes, as they neither require the full stability of normal epimorphisms nor the existence of all finite limits. 

Furthermore, Grandis, in his work \cite{GRANDIS13}, also considered another interesting property which is closely related to \wpnormality{}: a non-pointed generalisation of Noether's first isomorphism theorem. Suppose $\cat C$ is a category with a mono-coreflective subcategory $\tobj$ satisfying \zcref{wnorm-1,wnorm-2} of \zcref{wnorm} of \wnormality{}. Then, for any diagram of the form shown below on the left, where $m$, $n$ (and therefore $j$) are \npnormmono{}s, we can construct the corresponding diagram on the right, where $p$, $q$ and $r$ are the \npcokernel{}s of $m$, $n$ and $j$ respectively, and $\phi$ and $\psi$ are the induced maps making the diagram commutative.
\begin{equation}
\zcsetup{reftype=diagram} \label{Noet-iso-diagram}
\hskip\textwidth minus \textwidth
\begin{tikzcd}[baseline=(current bounding box.center)]
    & N\arrow[d, normmono, "n"]\arrow[dl, bend right, normmono, "j"']
    \\
    M\arrow[r, normmono, "m"'] & A
\end{tikzcd}
\hskip\textwidth minus \textwidth
\begin{tikzcd}[baseline=(current bounding box.center)]
N\arrow[r,equal]\arrow[d, "j"',normmono] & N\arrow[d, normmono, "n"]
\\
M\arrow[r, normmono, "m"']\arrow[d,normepi, "r"] & A\arrow[d,normepi, "q"]\arrow[r,normepi, "p"] & A/M\arrow[d,equal]
\\
M/N\arrow[r, "\phi"'] & A/N\arrow[r,"\psi"'] & A/M
\end{tikzcd}
\hskip\textwidth minus \textwidth
\end{equation}
The following proposition holds, and its core argument can be found in \cite{GRANDIS13}, though formulated in a different setting.
\begin{proposition}
\label{noether-np}
    Let $\cat C$ be a category and $\tobj$ a mono-coreflective subcategory satisfying \zcref{wnorm-1,wnorm-2} of \zcref{wnorm} of \wnormality{}. The following are equivalent.
\begin{enumerate}
    \item $\cat C$ and $\cat Z$ satisfy \zcref{wnorm-3} of \zcref{wnorm} (i.e.\ $\cat C$ is \wnormal{}).
    \item The following properties hold for $\cat C$:
    \begin{enumerate}[2a)]
        \item the composition of \npnormepi{}s is a \npnormepi{};
        \item in a situation like the one depicted in \zcref{Noet-iso-diagram}, the diagram on the right has \npexact{} rows (i.e. $M/N$ is a $\tobj$-normal subobject of  $A/N$ and $\frac{A/N}{M/N}\iso A/M$).
    \end{enumerate}
\end{enumerate}
    
\end{proposition}
\begin{proof}
    A similar equivalence is proved by Grandis in \cite{GRANDIS13} under the hypothesis of composable \npnormepi{}s, which we know is satisfied for \wnormal{} categories, since \npnormepi{}s are part of a factorisation system.
\end{proof}
\section{Examples}
\renewcommand*{\thesubsection}{\thesection.\alph{subsection}}\label{Examples}
\subsection{{Pointed examples and non-examples}}
In a pointed category $\cat C$, the subcategory $\tobj$ of zero objects is automatically epi-reflective and mono-coreflective, and $\cat C$ is \normal{} (respectively \wnormal{}) if and only if it is \pnormal{} (respectively \wpnormal{}).

\subsection{Trivial examples}
{\zcsetup{reftype=example} \label{trivial-ex}}
{
\renewcommand{\tobj}{\cat C}
For every category $\cat C$, the largest coreflective subcategory of $\cat C$ is $\cat C$ itself. If $\cat C$ has finite limits then it is \normal{}. \npnormepi{}s are exactly isomorphisms, and every map has a \nptrivial{} \npkernel{}.
}

If a category $\cat C$ has an initial object, then the full subcategory $\tobj$ of initial objects is the smallest coreflective subcategory of $\cat C$. If the initial object is strict (meaning that every morphism in $\cat C$ having the initial object as codomain is an isomorphism) and $\cat C$ is finitely complete, then $\cat C$ is \normal{}, and once again \npnormepi{}s coincide with isomorphisms, and every map has a \nptrivial{} \npkernel{}. In this case $\tobj$ is not reflective in $\cat C$ (unless $\tobj=\cat C$).
\subsection{Slices}
{\zcsetup{reftype=example} \label{ex-slices}}
Let $\cat C$ be a \normal{} category and $C$ an object in $\cat C$. It is well-known that when $\cat C$ is finitely complete, then so is the the slice category $\slice CC$. In particular, non-empty connected limits (such as pullbacks) are computed as in $\cat C$. The category $\slice Z C$ is a mono-coreflective subcategory of $\slice CC$, where the coreflection in $\slice ZC$ of an object $x\colon X\to C$ in $\slice CC$ is simply $(\comp{\cuni X}x\colon \corefl X\to C,\cuni X)$, with $(\corefl X,\cuni X)$ a coreflection of $X$ in $\cat Z$. It is easy to check that a map in $\slice CC$ is $(\slice ZC)$-trivial if and only if the same map, seen as a map in $\cat C$, is \nptrivial. 

By the above description of pullbacks and coreflections, it follows that $(\slice ZC)$-kernels can be computed as \npkernel{}s in $\cat C$.

Suppose now that $(h\colon H\to C, k)$ is the $(\slice ZC)$-kernel of a map $f$ in $\slice CC$ from $x\colon X\to C$ to $y\colon Y\to C$. By the previous paragraph, we have that $k$ is the \npkernel{} of $f$ in $\cat C$. We can thus consider its \npcokernel{} $(P,q)$ in $\cat C$. Since $\comp kf$ is $(\slice ZC)$-trivial in $\slice CC$, then it is also \nptrivial{} in $\cat C$. Therefore, there exists a map $f'$ in $\cat C$ such that $f=\comp q {f'}$. Call $p$ the composite $\comp{f'}y$.
\begin{equation*}
\begin{tikzcd}
    H\arrow[r, "k",normmono]\arrow[ddr, "h"', bend right] & X\arrow[r, "q", normepi]\arrow[d, "f"]\arrow[dd, bend right, "x"']& P\arrow[dl, "{f'}"{near end, xshift=-.5em}, bend left]\arrow[ddl, "p", bend left]
    \\
    & Y\arrow[d, "y"]
    \\
    &C
\end{tikzcd}
\end{equation*}
One can show that $(p\colon P\to C, q)$ is a $(\slice ZC)$-cokernel of $k$. Since every $(\slice ZC)$-cokernel is also the $(\slice ZC)$-cokernel of its $(\slice ZC)$-kernel, we can conclude that a map is a $(\slice ZC)$-normal epimorphism in $\slice CC$ if and only if the same map is a \npnormepi{} in $\cat C$. Stability of $(\slice ZC)$-normal epimorphisms in $\slice CC$  then follows from the stability of \npnormepi{}s in $\cat C$.  
\subsection{Functor categories}
{\zcsetup{reftype=example} \label{ex-fun-cat}}
Let $\cat C$ be a \normal{} category, and let $\cat B$ be any category. Then the functor category $\Fun BC$ is ($\Fun BZ$)-\pnormal. First, the subcategory $\Fun BZ$ is mono-coreflective in $\Fun BC$: if $S\colon \cat Z\to\cat C$ is a right adjoint to the inclusion $I\colon\cat Z\to \cat C$, then $\comp {(-)}S\colon \Fun BC\to\Fun BZ$ is right adjoint to the inclusion $\comp{(-)}I\colon\Fun BZ\to\Fun BC$. The component of the counit of this adjunction at a functor $F\colon \cat B\to \cat C$ is obtained as the whiskering $\comp F\varepsilon$, where $\varepsilon$ is the counit of the adjunction $I \adj S$. Since limits existing in $\cat C$ also exist in $\Fun BC$ and they are computed pointwise in $\Fun BC$, the latter is finitely complete, and $(\Fun BZ)$-kernels in $\Fun BC$ are also computed pointwise. Finally, one easily checks that $(\Fun BZ)$-cokernels of $(\Fun BZ)$-kernels exist and are computed pointwise as well.
\subsection{Categories of monomorphisms}
{
\newcommand{\monocl}{\mathscr S}
\newcommand{\monocat}{\cat S}
For any category $\cat E$, we denote by $\Arr E$ the category whose objects are the arrows of $\cat E$, and where morphisms from $x\colon X\to X_0$ to $y\colon Y\to Y_0$ are pairs $(f,f_0)$, with $f\colon X\to Y$ and $f_0\colon X_0\to Y_0$ such that $\comp fy=\mcomp{x,f_0}$.
\begin{definition}
    Let $\cat E$ be a category, and let $\monocl$ be a class of monomorphisms in $\cat E$ containing all identities.
    \begin{enumerate}[a)]
        \item Suppose $\cat E$ has pullbacks. We say that $\monocl$ is a \emph{\wscm{}} if the following properties hold.
        \begin{enumerate}[1.]
            \item For all $f$ and $g$, composable maps in $\cat E$, if $g$ and $\comp fg$ are in $\monocl$, then $f$ is in $\monocl$;
            \item the pullback of a map in $\monocl$ along any map in $\cat E$ is again in $\monocl$.
        \end{enumerate}
        \item Suppose $\cat E$ is finitely complete and call $\monocat$ the full subcategory of $\Arr E$ generated by $\monocl$. We say that $\monocl$ is a \emph{\scm{}} if $\monocat$ is closed under finite limits in $\Arr E$.
    \end{enumerate}
\end{definition}
\begin{lemma}
    Let $\cat E$ be a finitely complete category. Then every \scm{} $\monocl$ on $\cat E$ is \wstab{}.
\end{lemma}
\begin{proof}
    First, we show that, if $f\colon A\to B$ and $g\colon B\to C$ are maps in $\cat E$, with $h=\comp fg\in\monocl$ and $g$ a monomorphism, then $f$ is in $\monocl$. Notice that the following square is a pullback in $\Arr E$.
    \[
    \begin{tikzcd}
        f\arrow[r, "{(\id A,g)}"]\arrow[d, "{(f,\id B)}"'] &[1.8em] h\arrow[d, "{(h,\id C)}"]
        \\
        \id B\arrow[r, "{(g,g)}"'] &\id C
    \end{tikzcd}
    \]
    Since $\monocat$ is closed under limits in $\Arr E$, and $h$, $\id B$ and  $\id C$ are all in $\monocl$, we conclude that $f\in\monocl$ by closure under limits.

    Suppose we are now given a pullback square in $\cat E$ like the one below on the left, with $f\in \monocl$. Then the square on the right is a pullback square in $\Arr E$.
    \[
    \begin{tikzcd}
    A'\arrow[r, "a"]\arrow[d,"{f'}"'] & A\arrow[d, "f"] &[4em] f'\arrow[d, "{(f',\id {B'})}"']\arrow[r, "{(a,b)}"]& f \arrow[d, "{(f,\id B)}"]
    \\
    B'\arrow[r, "b"'] & B & \id {B'} \arrow[r,"{(b,b)}"']&\id B
    \end{tikzcd}
    \]
    Since $f$, $\id B$ and $\id{B'}$ are all in $\monocl$, we conclude that $f'$ is also in $\monocl$ by closure under limits.
\end{proof}
\begin{example}
    In a category $\cat E$ with pullbacks, the following classes of monomorphisms are all \wstab{}: 
    \begin{enumerate*}
        \item all monomorphisms; \item strong monomorphisms; \item regular monomorphisms; \item normal monomorphisms, when $\cat E$ is pointed; \item more in general, \npnormmono{}s for any reflective subcategory $\tobj$ of $\cat E$ (reflectivity ensures that identity maps are \npnormmono{}s). 
    \end{enumerate*}
    
    If $\cat E$ is finitely complete, the classes of monomorphisms and of strong monomorphisms, and any class $\mathscr R$ of monomorphisms in $\cat E$ which is part of a factorisation system $(\mathscr L,\mathscr R)$ on $\cat E$ are all \scms{}.
\end{example}
\begin{proposition}
    Let $\cat E$ be any category and 
    let $\monocl$ be a class of monomorphisms in $\cat E$ containing all identity maps. Call $\tobj$ the full subcategory of $\Arr E$ whose objects are all isomorphisms in $\cat E$, and call $\monocat$ the full subcategory of $\Arr E$ generated by $\monocl$.
    \begin{enumerate}
        \item\label{wn-monocat} If $\cat E$ has pullbacks and $\monocl$ is a \wscm{}, then $\monocat$ is \wnormal{}.
        \item\label{n-monocat} If $\cat E$ is finitely complete and $\monocl$ is a \scm{}, then $\monocat$ is \normal{}.
    \end{enumerate}
\end{proposition}
\begin{proof}
As $\monocl$ contains all identities, it is straightforward to verify that $\tobj$ is both mono-coreflective and epi-reflective in $\monocat$. The computation of \npkernel{}s and \npcokernel{}s in similar contexts has been studied in various sources (such as \cite{GRANDIS13,MARKI13}), and it is easy to check that, when $\cat E$ has pullbacks and $\monocl$ is a \wscm{}, \npkernel{}s and \npcokernel{}s of \npkernel{}s exist in $\monocat$. Moreover, \npnormmono{}s and \npnormepi{}s in $\monocat$ can be characterised (up to isomorphism) as morphisms of the form $(\id{}, s)$ and $(s, \id{})$, respectively, with $s\in\monocl$. Morphisms with \nptrivial{} \npkernel{} are commutative square which are pullbacks in $\cat E$.

A direct check shows that all the relevant limits in $\monocat$ required for \wnormality{} or \normality{} 
can be computed level-wise in $\cat E$. As a result, the stability conditions for \npnormepi{}s follow immediately.
\end{proof}
}
\subsection{Groupoids}
\newcommand{\catGrpd}{\cats{Grpd}}In the category $\catGrpd$ of small groupoids and functors consider the full subcategory $\tobj$ of discrete groupoids (i.e.\ $\tobj\iso\catSet$). Then, with this definition of $\tobj$, $\catGrpd$ is \normal{}. The construction of \npkernel{}s and \npcokernel{}s is described in \cite{GRANDIS13}; a functor $F\colon\cat G\to\cat H$ is a \npnormepi{} if and only if it is strictly surjective (on arrows) and for any arrows $g$ and $g'$ in $\cat G$ such that $Fg=Fg'$ there exist arrows $u$ and $u'$ in $\cat G$ such that $Fu$ and $Fu'$ are identity morphisms in $\cat H$ and such that the compositions $\comp gu$ and $\comp {u'}{g'}$ are defined and equal (cf.\ the analogous condition \ref{cmon-normepi} for commutative monoids). It is a straightforward check that this class of morphisms is pullback-stable in $\catGrpd$.
\subsection{Categories of relations}
\newcommand{\kp}[1]{\textnormal{K}_{#1}}\newcommand{\intersect}{\wedge}
Let $\cat R$ denote any of the following categories of relations on sets, where morphisms are relation-preserving functions:
\begin{enumerate*}[(1)]
    \item\label{refl-rel} the category of reflexive relations,
    \item\label{preord-rel}the category of preorders,
    \item \label{equiv-rel}the category of equivalence relations.
\end{enumerate*}

Let $\tobj$ be the full subcategory of $\cat R$ consisting of sets equipped with the discrete equivalence relation. For any morphism $f\colon (X,\rho)\to (Y,\rho)$ in $\cat R$, its \npkernel{} is given by $\bigl((X,\kp f\intersect \sigma),\id X\bigr)$, where $\kp f$ denotes the kernel pair of $f$ (in $\catSet$) and $\intersect$ denotes the intersection of relations. On the other hand, its \npcokernel{} is given by $\bigl(( Y_0, \sigma_0), p\bigr)$, where $p\colon Y\to Y_0$ is the quotient in $\catSet$ of $Y$ by the equivalence relation on $Y$ generated by $f(\rho)$, and $\sigma_0$ is the relation of type $\cat R$ generated by $p(\sigma)$ on $ Y_0$. By this construction, it follows that a morphism $f\colon (X,\rho)\to(Y,\sigma)$ in $\cat R$ is a \npnormepi{} if and only if $f$ is surjective, $\sigma$ is the relation of type $\cat R$ generated by $f(\rho)$ on $Y$ and the kernel pair $\kp f$ of $f$ is the smallest equivalence relation on $X$ containing $\kp f \intersect \rho$.

It can be shown that when $\cat R$ is either the category of reflexive relations or of equivalence relations (cases \ref{refl-rel} and \ref{equiv-rel}), then it is \normal{}. The key point in proving the pullback-stability of \npnormepi{}s is that, under the hypotheses of \zcref{refl-rel} or \ref{equiv-rel} above, the image $f(\rho)$ of a relation $\rho$ of type $\cat R$ under any \npnormepi{} $f\colon (X,\rho)\to(Y,\sigma)$ in $\cat R$ is automatically a relation of type $\cat R$: namely, a reflexive relation in case \ref{refl-rel} and an equivalence relation in case \ref{equiv-rel}.

The category of preordered sets, instead, is not even \wnormal{}. To prove this, consider the morphisms of preordered sets
\[
\left\{
\begin{tikzcd}[row sep = 1em, column sep =1em]
    1\arrow[r]&2
    &2'\arrow[l]\arrow[r]&3
\end{tikzcd}
\middle\}\,\tikz{\draw[->] (0,0) to node[midway, above]{\scriptsize$f$} (1,0);}\,\middle\{\begin{tikzcd}[row sep = 1em, column sep =1em]
    1\arrow[r,bend left]&2\arrow[l, bend left]
    \end{tikzcd}\right\}
\]
defined by $f(1)=f(3)=1$ and $f(2)=f(2')=2$. Using the above constructions of \npkernel{}s and \npcokernel{}s, the morphism $f$ can be explicitly factorised as $f=\comp qm$, with $q$ the \npcokernel{} of the \npkernel{} of $f$, and $m$ the induced map. Computing the \npkernel{} of $m$, one easily finds that it is not \nptrivial{}, and therefore $\cat R$ cannot be \wnormal{} (nor \normal{}).
\subsection{Inverse commutative monoids}
A commutative monoid $M$ is called an \emph{inverse commutative monoid} if for every $x\in M$ there exists an element $x\inv\in M$, called the \emph{inverse} of $x$, such that
\begin{equation}
\label{inv-axioms}
    xx\inv x=x \quad\textnormal{and}\quad x\inv x x\inv=x\inv.
\end{equation}
(in this subsection we shall use multiplicative notation). Thanks to commutativity, such an inverse is  unique (see \cite{LAWSON}). We denote by $\catICM$ the category of inverse commutative monoids and monoid morphisms (inverses are then automatically preserved).

\begin{remark}
\label{inv-mon-are-om}
     Inverse commutative monoids can be seen as commutative $\Omega$-monoids with $\Omega$ consisting of a unique unary operation $(-)\inv$ satisfying the axioms in \ref{inv-axioms} above. It immediately follows that $\catICM$ is prenormal (see \zcref{omega-cmon}).
\end{remark}

Despite $\catICM$ being pointed, we are now interested in a different class of trivial objects. In a monoid, an element $e$ is said to be \emph{idempotent} if $e^2=e$. A monoid is said to be \emph{idempotent} if all of its elements are idempotent. We denote by $\catECM$ the category of idempotent commutative monoids with monoid morphisms. Clearly, every idempotent commutative monoid is an inverse monoid, where the inverse of an element $e$ is $e$ itself. We now show that $\catICM$ is prenormal with respect to $\catECM$.

We have functors
\[
\begin{tikzcd}
 \catECM \rar[hookrightarrow,bend right, start anchor=east, end anchor=west, yshift=-.2em, "I"']& \catICM\arrow[l,bend right, "E"', start anchor=west, end anchor= east, yshift=.2em],
 \end{tikzcd}
\]
where $I$ denotes the inclusion functor, while $E$ denotes the restriction to idempotents (for any inverse monoid $M$, the subset $E(M)$ of idempotent elements is an idempotent submonoid, and monoid morphisms obviously preserve idempotents). 
The functor $E$ is both left and right adjoint to $I$. The unit $\unin$ of the adjunction $E\adj I$ and the counit $\cunn$ of the adjunction $I\adj E$ are given by
\[
\begin{tikzcd}
    M\rar{\uni M} & E(M) &[2em]E(M)\rar[hookrightarrow]{\cuni M}&M
    \\[-3em]
    x\rar[mapsto] &xx\inv &e\rar[mapsto] &e
\end{tikzcd}
\]
for all inverse commutative monoids $M$. Clearly, $\uni M$ is an epimorphism and $\cuni M$ a monomorphism (in fact, $\mcomp{\cuni M, \uni M}=\id{E(M)})$. Taking $\tobj =\catECM$, we have the following characterisations.
\begin{itemize}[leftmargin=*]
    \item A \npnormmono{} is, up to isomorphism, the inclusion of a normal inverse submonoid $A\inclusion M$, that is a submonoid $A$ closed under inverses and conjugation and containing all idempotents of $M$ (see \cite{LAWSON}).
    \item A morphism $f\colon M\to N$ is a \npnormepi{} if and only if it is surjective and the restriction $E(f)$ to idempotents is an isomorphism. Indeed, every \npnormepi{} is an epimorphism, and its $\tobj$-reflection is always an isomorphism. Vice versa, suppose $f\colon M\to N$ is surjective and that $E(f)$ is an isomorphism. Consider its \npkernel{} $k\colon K\inclusion M$, consisting of elements $x\in M$ such that $f(x)$ is idempotent. Let $g\colon M\to P$ be a morphism in $\catICM$ such that $\comp kg$ is \nptrivial{} (i.e. $g(x)$ is idempotent for all $x\in K$). We must show that if $f(x)=f(y)$, then $g(x)=g(y)$, so that $g$ factors uniquely through $f$. 
    A simple computation shows that
    \begin{equation*}
    \label{eq-idmptn}
       xx\inv=xy\inv x\inv y=yy\inv,
    \end{equation*}
    as they are idempotents and their images via $f$ are equal.
    Since $f(xy\inv)$ is obviously idempotent, we have $xy\inv\in K$, so $g(xy\inv)$ is itself idempotent. We can thus write
    \begin{multline*}
        g(x)=g(xx\inv x)=g(xy\inv x\inv y x)=g(xy\inv)g(x\inv y)g(x)\\=g(x\inv y)g(x)=g(xx\inv y)=g(yy\inv y)=g(y).
    \end{multline*}
\end{itemize}

With these characterisations in place, it is straightforward to verify that the category of inverse commutative monoids is \normal{}. Indeed, $\catICM$ is a variety of universal algebra, and is therefore complete and cocomplete. We have seen that $\tobj=\catECM$ is both mono-coreflective and epi-reflective. In particular, all \npcokernel{}s exist, and can be obtained as pushout along $\tobj$-reflections (dual of \zcref{monocoref}). Now, let $f'$ be the pullback of a \npnormepi{} $f$ along an arbitrary morphism. Clearly, $f'$ is surjective (since pullbacks are computed as in $\catSet$ and thus preserve surjective maps). Moreover, because $E$ is a right adjoint, it preserves pullbacks. Hence,  $E(f')$ appears as the pullback of the isomorphism $E(f)$, and is therefore itself an isomorphism. 
We conclude that $f'$ is a \npnormepi{}.
\begin{remark}
    This example was inspired by the recent talk \cite{ULO}, where the speaker observed that for any regular epimorphism $f \colon X \to Y$ in the category of inverse monoids, the pullback of $f$ along $E(Y)\inclusion Y$ is also a pushout. This observation led us to wonder whether the category might be prenormal.
\end{remark}
\begin{remark}
    Note that, with the exception of \zcref{inv-mon-are-om} which involves pointedness, everything in this subsection also applies to inverse commutative semigroups. See \cite{LAWSON} for more details.
\end{remark}
\subsection{Preordered groups}
Let $\catOrdGrp$ be the category whose objects are preordered groups, namely groups $G$ equipped with a preorder relation $\le$ satisfying $xx'\le yy'$ for all $x\le x'$ and $y\le y'$ in $G$,
and whose arrows are monotone group homomorphisms. This category is equivalent to the one having as objects the pairs $(G,M)$, where $G$ is a group and $M$ is a submonoid of $G$ closed under conjugation, and morphisms $(G,M)\to (H,N)$ given by a group homomorphisms $f\colon G\to H$ such that $f(M)\subseteq N$. Many remarkable properties of this category have been studied in \cite{CLEMENTINO19}. In particular, it was proved that $\catOrdGrp$ is normal (in the sense of \cite{NORMAL}), and therefore \pnormal{} as well.

Alongside the usual class of zero objects, in \cite{MICHEL}, the authors consider another class of trivial objects in $\catOrdGrp$, namely the class $\tobj$ of pairs of the form $(G,0)$, with $G$ any group (these pairs correspond exactly to groups equipped with the discrete preorder). In what follows, we prove that $\catOrdGrp$ is \normal{}.

The subcategory $\tobj$ is mono-coreflective, with the coreflection of a pair $(G,M)$ given by the inclusion $\id G\colon (G,0)\to(G,M)$. Since pullbacks are computed level-wise, we can characterise \npnormmono{}s as maps of the form $\id G\colon (G,M)\to(G,N)$, where $M$ and $N$ are  submonoids of the group $G$ closed under conjugation and with $M\subseteq N$. We have that the \npcokernel{} of such a \npnormmono{} is given by $q\colon (G, N)\to (Q,P)$, where $q\colon G\to Q$ is the quotient in $\catGrp$ of $G$ by the subgroup $M'$ generated by $M$ in $G$, and $P$ is simply $q(N)$ (i.e.\ the regular image of $q\vert\colon N\to Q$ in $\catMon$). Note that $M'$ is explicitly given by the set of finite products $x_1x_2\cdots x_n$ in $G$, with $x_i\in M$ or $x_i\inv\in M$ for all $i\in\{1,\dots, n\}$; this subgroup is automatically normal in $G$, since $M$ is closed under conjugation. The universal property of $((Q,P),q)$ can be proved directly using, in particular, the left lifting property of regular epimorphisms with respect to monomorphisms in $\catMon$. We can then characterise \npnormepi{}s as maps $f\colon (G,M)\to (H,N)$ such that $f(G)=H$, $f(M)=N$ and such that the kernel of $f\colon G\to H$ is generated in $G$ by the kernel of the monoid morphism $f\vert\colon M\to N$ (using the above description of \npcokernel{}s, one can prove that this is the \npcokernel{} of $\id G\colon (G,L)\to (G,M)$, with $L$ being the kernel of the monoid morphism $f\vert\colon M\to N$). Finally, one can show that this class of maps is pullback-stable, using, in particular, the explicit description of the (normal) subgroup  generated by a subset provided above. 

\section*{Acknowledgments}
\addcontentsline{toc}{section}{Acknowledgments}
This research was conducted while both authors were affiliated with INdAM -- Istituto Nazionale di Alta Matematica ‘Francesco Severi’, Gruppo Nazionale per le Strutture Algebriche, Geometriche e le loro Applicazioni (GNSAGA).
\phantomsection
\addcontentsline{toc}{section}{References}
\printbibliography
\end{document}